\documentclass[preprint,12pt]{elsarticle}

\usepackage{amssymb}
\usepackage{amsmath}
\usepackage{amsthm}

\usepackage{url}
\usepackage{doi}

\usepackage{etoolbox} 
\usepackage{xcolor} 
\usepackage{xspace}
\usepackage{amsmath}
\usepackage{bbold}
\usepackage[all]{xy}
\usepackage{mathtools}
\usepackage[mathscr]{euscript}

\def\bfd{\begin{color}{orange}}
\def\efd{\end{color}}

\def\cf{cf.~} 
\newcommand{\refItem}[2]{\cref{#1}(\ref{#1:#2})}

\def\ple#1{\ensuremath{{\langle #1 \rangle }}} 

\def\vuoto{}
\newcommand{\FUN}[4]{\ensuremath{{#2}#1{#3}\rightarrow{#4}}}
\newcommand{\fun}[3]{\relax\def\testa{#1}\relax\ifx\testa\vuoto
  \relax\FUN{}{}{{#2}}{{#3}}\else\relax\FUN{:}{{#1}}{{#2}}{{#3}}\fi}

\newcommand{\id}{\mathsf{id}}
\newcommand{\PW}[2][]{\ensuremath{\mathop{\mathscr{P}_{#1}{#2}}}}
\newcommand{\pw}[1]{\relax\def\testa{#1}\relax\ifx\testa\vuoto
  \relax\PW{}\else\relax\PW{\left(#1\right)}\fi}
\newcommand{\fpw}[1]{\relax\def\testa{#1}\relax\ifx\testa\vuoto
  \relax\PW[\omega]{}\else\relax\PW[\omega]{\left(#1\right)}\fi}

\newcommand{\eqc}[2][]{[#2]_{#1}} 

\newcommand{\N}{\mathbb{N}}
\newcommand{\R}{\mathbb{R}} 
\newcommand{\RPos}{\R_{\ge 0}} 

\newcommand{\PP}{\pw{}} 
 
\newcommand{\Rel}{\mathsf{Rel}}

\def\tt{\ensuremath{\textsf{t\kern-.3ex t}}}
\def\ff{\ensuremath{\textsf{f\kern-.3ex f}}}

\let\ForalL\forall \def\Forall#1.{\ForalL_{#1}}
\let\ExistS\exists \def\Exists#1.{\ExistS_{#1}}

\DeclareFontFamily{OT1}{pzc}{}
\DeclareFontShape{OT1}{pzc}{m}{it}{<->s*[1.30]pzcmi7t}{}
\DeclareMathAlphabet{\mathpzc}{OT1}{pzc}{m}{it}
\def\ct#1{\ensuremath{\mathpzc{#1}}}
\def\Ct#1{\ensuremath{\mathbf{#1}}}

\def\op{^{\mbox{\normalfont\scriptsize op}}}

\newcommand{\bop}[1]{(#1\times#1)\op} 
\def\blank{\mathchoice{\mbox{--}}{\mbox{--}}
{\mbox{\scriptsize--}}{\mbox{\tiny--}}}

\newcommand{\CC}{\ct{C}\xspace}

\newcommand{\D}{\ct{D}\xspace}

\newcommand{\oneAr}[3]{\ensuremath{{#1}:{#2}\rightarrow{#3}}}
\newcommand{\twoAr}[3]{\ensuremath{{#1}:{#2}\Rightarrow{#3}}}
\def\tdot{\textbf{.}}
\def\lsta{\vrule depth4pt width0pt}
\def\lstb{\vrule height5pt width0pt}
\def\arnta[#1]{\ar[#1]|-*=0[@]{\lsta\tdot}}
\def\arntb[#1]{\ar[#1]|-*=0[@]{\lstb\tdot}}
\newcommand{\nt}[3]{\ensuremath{{#1}:{#2} \stackrel{\makebox{\kern-.3ex\tdot}}\rightarrow{#3}}}
\newcommand{\lnt}[3]{\ensuremath{{#1}:{#2} \stackrel{\makebox{\kern-.3ex\tdot}}\rightarrow_l{#3}}}

\newcommand{\Id}{\mathsf{Id}}
\newcommand{\ID}{\mathrm{Id}}

\newcommand{\Set}{\ct{Set}\xspace}

\newcommand{\Pos}{\ct{Pos}\xspace}

\newcommand{\RDtn}{\Ct{RD}\xspace}

\newcommand{\QRDtn}{\Ct{QRD}\xspace}

\newcommand{\EQRDtn}{\Ct{EQRD}\xspace}

\newcommand{\order}{\leq}

\newcommand{\RDoc}{R}
\newcommand{\BM}{\mathsf{BM}}

\newcommand{\SDoc}{S} 
\newcommand{\reidx}[1]{_{#1}}

\newcommand{\fn}[1]{\widehat{#1}}
\newcommand{\lift}[1]{\overline{#1}}

\newcommand{\coveredarr}{c}

\newcommand{\relr}{\alpha}
\newcommand{\rels}{\beta}
\newcommand{\relt}{\gamma} 
\newcommand{\eqrelr}{\rho}
\newcommand{\eqrels}{\sigma}
 
\newcommand{\rid}{\mathsf{d}}
\newcommand{\rcomp}{\mathop{\mathbf{;}}}
\newcommand{\rconv}{^{\bot}} 
\newcommand{\rdconv}{^{\bot\bot}} 
\newcommand{\gr}[1]{\Gamma_{#1}}

\newcommand{\EQRFun}{\mathrm{U_{eq}}}
\newcommand{\REQFun}{\mathrm{EQ}}
\newcommand{\EQR}[1]{(#1)^{eq}} 
\newcommand{\EQC}[1]{\ensuremath{\ct{EQ}_{#1}}\xspace}

\newcommand{\VRel}[1]{{#1}\text{-}\mathsf{Rel}}

\newcommand{\Car}[1]{|#1|}

\newcommand{\dist}{\delta}
\newcommand{\Met}{\ct{Met}}

\def\RB#1{\mathchoice
  {\rotatebox[origin=c]{180}{$#1$}}
  {\rotatebox[origin=c]{180}{$#1$}}
  {\rotatebox[origin=c]{180}{$\scriptstyle#1$}}
  {\rotatebox[origin=c]{180}{$\scriptscriptstyle#1$}}}
\def\Ex{\RB{E}\kern-.3ex}
\def\Al{\RB{A}\kern-.6ex}

\newcommand{\Span}[1]{\mathsf{Spn}^{#1}}
\newcommand{\JSpan}[1]{\mathsf{JSpn}^{#1}}
\newcommand{\spn}[5]{\ensuremath{#1 \xleftarrow{#2} #3 \xrightarrow{#4} #5}}

\newcommand{\jmSpan}[1]{\mathsf{JmSpn}^{#1}}

\newcommand{\mnd}{\mathbb{T}} 
\newcommand{\Lmnd}{\mathbb{L}}
\newcommand{\Kmnd}{\mathbb{K}}

\newcommand{\Proj}[1]{\ct{P}_{#1}}
\newcommand{\RProj}[1]{\mathsf{Proj}^{#1}}

\newcommand{\app}{.}

\usepackage[capitalize]{cleveref}

\theoremstyle{theorem}
\newtheorem{theorem}{Theorem}[section]
\newtheorem{lemma}[theorem]{Lemma}
\newtheorem{proposition}[theorem]{Proposition}
\newtheorem{corollary}[theorem]{Corollary} 

\theoremstyle{definition}
\newtheorem{definition}[theorem]{Definition}
\newtheorem{remark}[theorem]{Remark} 
\newtheorem{example}[theorem]{Example} 

  \crefname{fig}{Figure}{Figures}
  \crefname{thm}{Theorem}{Theorems}
  \crefname{lem}{Lemma}{Lemmas}
  \crefname{def}{Definition}{Definitions}
  \crefname{prop}{Proposition}{Propositions}
  \crefname{cor}{Corollary}{Corollaries}
  \crefname{ex}{Example}{Examples}
  \crefname{rem}{Remark}{Remarks}
  \crefname{asm}{Assumption}{Assumptions}

\journal{JPAA}

\begin{document}

\begin{frontmatter}

\title{Projective covers, doctrines of algebras and the relational quotient completion} 

\author[fd]{Francesco Dagnino}
\ead{francesco.dagnino@unige.it}

\author[fp]{Fabio Pasquali}
\ead{fabio.pasquali@unimi.it}

\begin{abstract}
The extensional quotient completion of relational doctrines provides  a common generalization of both 
the exact completion of categories with weak finite limits and 
the elementary quotient completion of existential elementary doctrines. 
In this paper, we study projective objects  in relational doctrines with quotients, characterizing those obtained though the extensional quotient completion as those admitting a projective cover. 
We apply this result to doctrines of algebras for monads on relational doctrines with quotients, describing in which cases these arise as the extensional quotient completion of their restriction to (appropriate subcategories of) free algebras. 
This extends a similar result for monadic categories over exact ones, 
covering also more examples such as monads over the category of metric spaces giving rise to variants of quantitative algebras. 
\end{abstract}


\end{frontmatter}

\section{Introduction}
\label{sect:intro}

Every algebra is a quotient of a free algebra. This fundamental result establishes that any algebraic structure, like a monoid, group, or ring, can be presented by generators and relations.
Roughly, this means that, in order to describe an algebra, we just have to pick a set of basic elements, construct the free algebra of formal expressions built upon them, and then impose equations that define the relationships between these expressions. The desired algebra is then obtained by taking the quotient of that free algebra by the specified equations.
Furthermore, we can present homomorphisms of algebras in a similar way. 
That is, given presentations of two algebras, we can describe a homomorphism between them 
as a homomorphism between the associated free algebras which preserves the specified equations. 

This fact finds a general and uniform explanation using the language of category theory, in particular, the notions of an (Barr) exact category \cite{Barr71} and the associated exact completion \cite{CarboniA:freecl, CarboniA:regec}.
Exact categories are categories with finite limits and a well-behaved notion of quotient for internal equivalence relations. Among them are abelian categories and categories of algebras over $\Set$ (the category of sets and functions).
The exact completion is a construction that turns any category with weak finite limits into an exact one. Its objects are pairs of an object and a (pseudo)equivalence relation in the original category, and its arrows are classes of arrows\footnote{Two parallel arrows belong to the same class if, intuitively, they map equivalent elements of their domain to equivalent elements of their codomain.} between the underlying objects that preserve the equivalence relations.
A fundamental result about the exact completion states that the category of algebras for a monad on an exact category where (regular) epimorphisms split  is the exact completion of its subcategory of free algebras \cite{Vitale94}. This result encompasses categories of algebras over $\Set$ and captures, in categorical terms, the fact that algebras can be presented by generators and relations.

The distinguishing feature of exact categories that makes this analysis possible is the ability of computing quotients. Relational doctrines \cite{DagninoP23, DagninoP25tac,DagninoP25apal} have been recently introduced as a framework supporting a very general notion of quotient, together with an associated universal construction dubbed \emph{relational quotient completion}. \footnote{This is inspired by and generalizes the elementary quotient completion of an existential elementary doctrine by Maietti and Rosolini~\cite{MaiettiME:eleqc}.} As specific instances of these notions one recovers exact categories and the exact completion, as well as many other categories of equivalence relations, such as equilogical spaces and categories of assemblies for partial combinatory algebras \cite{lmcs:4302}, but also categories of metric spaces and of normed vector spaces.

In this paper, we study to which extent the aforementioned results about categories of algebras and the exact completion extend to relational doctrines of algebras and the relational quotient completion. This allows us to cover a wider range of algebras and presentations, including, for instance, variants of quantitative algebras~\cite{MardarePP16, MardarePP17, Adamek22}.
To achieve this, we study a notion of projective object for relational doctrines with quotients and characterize the relational doctrines obtained through the relational quotient completion as those admitting a projective cover. 
This generalizes the well-known theorem in \cite{CarboniA:regec} that characterizes the exact categories that are an exact completion as those with enough regular projectives.
Then, given that relational doctrines form a 2-category, we can consider monads on relational doctrines defined as monads in such a 2-category \cite{Street72}, which induce a relational doctrine of algebras \cite{DagninoP25tac}. We show that every projective cover of a relational doctrine determines a projective cover of the relational doctrine of algebras for a monad on it, consisting of those free algebras with projective generators. This shows that relational doctrines of algebras arise as relational quotient completions of a doctrine on a suitable subcategory of free algebras. Moreover, when the relational doctrine enjoys a form of the Rule of Choice, analogous to the requirement that epimorphisms split, we have that all free algebras are projective, thus obtaining a result analogous to the one for the exact completion.

The paper is organized as follows.
In \cref{sect:prelim} we recall some basic properties of relational doctrines and the relational quotient completion.
\cref{sect:proj} provides the characterization of those doctrines arising from the relational quotient completion via projective covers,  showing also that the characterisation in \cite{CarboniA:regec} of the exact completion is an instance of our theorem for a specific class of relational doctrines. 
Finally, in \cref{sect:algebra} we study doctrines of algebras for monads on relational doctrines, describing in which cases these arise as relational quotient completions of their restriction to (subcategories of) free algebras, again  extending the analogous result for the exact completion. 
To this end, we also develop some basic results about factorization systems and choice principles in relational doctrines with quotients.

\section{Preliminaries}
\label{sect:prelim} 
Relational doctrines \cite{DagninoP23,DagninoP25tac} were introduced as a functorial description of (the core fragment of) the calculus of relations \cite{Tarski41}, 
that is, the fragment whose operations are only the identity, the composition and the converse. 
They can be synthetically defined as the faithful framed bicategories \cite{Shulman08} with a dagger or, more explicitly, as follows.

\begin{definition}[Relational Doctrine]\label[def]{def:rel-doc}
A \emph{relational doctrine} consists of a base category \CC and a functor \fun{\RDoc}{\bop\CC}{\Pos} such that for every triple of objects $X,Y,Z$ in \CC there is a monotone function \[\fun{\blank\rcomp\blank}{\RDoc(X,Y)\times\RDoc(Y,Z)}{\RDoc(X,Z)}\] and an element $\rid_X \in \RDoc(X,X)$ such that 

\begin{align*} 
\relr\rcomp(\rels\rcomp\relt) &= (\relr\rcomp\rels)\rcomp\relt 
& 
\rid_X\rcomp\relr &= \relr 
&
\relr\rcomp\rid_Y &= \relr 
\\
(\relr\rcomp\rels)\rconv &= \rels\rconv \rcomp \relr\rconv 
&
\rid_X\rconv &= \rid_X
&
\relr\rdconv &= \relr 
\end{align*} 
and such that for all $\relr\in\RDoc(X,Y)$, $\rels\in\RDoc(Y,Z)$ and 
\fun{f}{A}{X}, \fun{g}{B}{Y} and \fun{h}{C}{Z}  in \CC it holds that
\[\RDoc\reidx{f,g}(\relr)\rcomp\RDoc\reidx{g,h}(\rels) \order \RDoc\reidx{f,h}(\relr\rcomp\rels)
\qquad
\rid_X \order \RDoc\reidx{f,f}(\rid_Y)
\qquad
(\RDoc\reidx{f,g}(\relr))\rconv \order \RDoc\reidx{g,f}(\relr\rconv)\] 
\end{definition}

The element $\rid_X$ is the \emph{identity} or \emph{diagonal} relation on $X$, 
$\relr\rcomp\rels$  is the \emph{relational composition} of $\relr$ followed by $\rels$, and 
$\relr\rconv$ is the \emph{converse} of the relation $\relr$.  
Note that all relational operations are lax natural transformations, but 
the operation of taking the converse, being  an involution, is actually strictly natural. 
The requirement of lax naturality becomes clear if one looks at the examples. 
For instance, in the paradigmatic example of set-theoretic relations, given by the relational doctrine \fun{\Rel}{\bop\Set}{\Pos} 
where $\Rel(X,Y)=\PP(X\times Y)$ and  $\Rel(f,g)=(f\times g)^{-1}$; it is easy to check that 
the inclusion
$\Rel\reidx{f,g}(\relr)\rcomp\Rel\reidx{g,h}(\rels)
\subseteq\Rel\reidx{f,h}(\relr\rcomp\rels)$
is an equality if and only if $g$ is surjective and, similarly, the inclusion $\rid_X 
\subseteq \RDoc\reidx{f,f}(\rid_Y)$ is an equality if and only if $f$ is injective.

Some examples of relational doctrines can be found in \cite{DagninoP23,DagninoP25tac,DagninoP25apal}, here we just consider those that will play a relevant role in the paper.

\begin{example} \label[ex]{ex:rel-doc} 
\begin{enumerate}
\item\label{ex:rel-doc:span}
Let \CC be a category with weak pullbacks. 
Denote by $\Span\CC(X,Y)$ the poset reflection of the preorder whose objects are spans in \CC between $X$ and $Y$ and where 
$\spn{X}{p_1}{A}{p_2}{Y} \order \spn{X}{q_1}{B}{q_2}{Y}$ if and only if there is an arrow \fun{f}{A}{B} such that 
$p_1 = q_1 \circ f $ and $p_2 = q_2 \circ f$. 
Given a span $\relr = \spn{X}{p_1}{A}{p_2}{Y}$ and arrows \fun{f}{X'}{X} and \fun{g}{Y'}{Y} in \CC, define 
$\Span\CC\reidx{f,g}(\relr) \in \Span\CC(X',Y')$ by one of the following equivalent diagrams: 
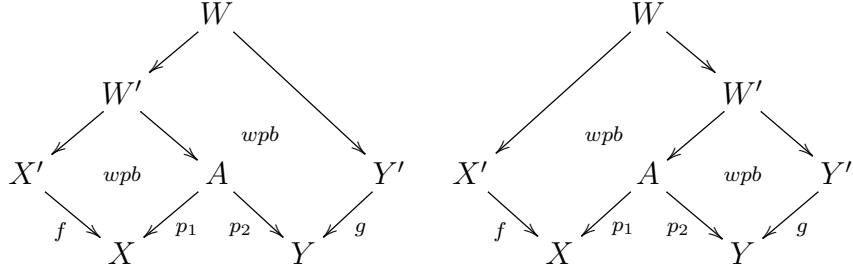
\[
\vcenter{\xymatrix@C=3ex@R=3ex{
&& W \ar[ld] \ar[rrdd] \ar@{}[rddd]|{wpb} && \\
& W'\ar[ld] \ar[rd] \ar@{}[dd]|{wpb} &&& \\  
X' \ar[rd]_-{f} && A \ar[ld]^-{p_1} \ar[rd]_-{p_2} && Y' \ar[ld]^-{g} \\ 
& X && Y 
}} 
\quad 
\vcenter{\xymatrix@C=3ex@R=3ex{
&& W \ar[rd] \ar[lldd] \ar@{}[lddd]|{wpb} && \\
&&& W'\ar[rd] \ar[ld] \ar@{}[dd]|{wpb} & \\  
X' \ar[rd]_-{f} && A \ar[ld]^-{p_1} \ar[rd]_-{p_2} && Y' \ar[ld]^-{g} \\ 
& X && Y 
}} 
\]
The functor \fun{\Span\CC}{\bop\CC}{\Pos} is a relational doctrine where, for 
$\relr = \spn{X}{p_1}{A}{p_2}{Y}$ and 
$\rels = \spn{Y}{q_1}{B}{q_2}{Z}$ it is 
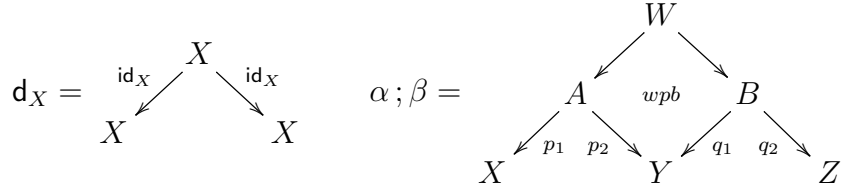
\[
\rid_X = \vcenter{\xymatrix@R=3ex@C=3ex{
& X \ar[ld]_-{\id_X} \ar[rd]^-{\id_X} & \\ 
X && X
}} \qquad 
\relr \rcomp \rels = \vcenter{\xymatrix@R=3ex@C=3ex{
&& W \ar[ld] \ar[rd] \ar@{}[dd]|{wpb} && \\ 
& A \ar[ld]^-{p_1}\ar[rd]_-{p_2} && B \ar[ld]^-{q_1} \ar[rd]_-{q_2} & \\ 
X && Y && Z 
}}
\]
while $\relr\rconv=  \spn{Y}{p_2}{A}{p_1}{X}$%
\item\label{ex:rel-doc:jmspans}
Suppose \CC is regular  and 
denote by $\jmSpan\CC(X,Y)$ the subposet of $\Span\CC(X,Y)$ on jointly monic spans. This extends to a functor  \fun{\jmSpan\CC}{\bop\CC}{\Pos} which is a relational doctrine, where the identity relations and the converses are defined as in \cref{ex:rel-doc:span}, while to compute the relational composition of two jointly monic spans $\relr$ and $\rels$ it suffices to observe that any spans factors through a jointly monic one.
\item\label{ex:rel-doc:vrel}
Let $\RPos = \ple{[0,\infty],\ge,+,0}$ be the quantale of extended non-negative real numbers  (here addition and subtraction in $[0,\infty]$ are extensions of those in $[0,\infty)$ by the rules $\infty+a=a+\infty=\infty-a=\infty$ and $a-\infty=0$ see \cite{Lawvere73}).
A \emph{metric relation} between sets $X$ and $Y$ is just a function \fun{\alpha}{X\times Y}{[0,\infty]}. 
The functor \fun{\VRel\RPos}{\bop\Set}{\Pos} maps two sets $X$ and $Y$ to 
$\VRel\RPos(X,Y) = [0,\infty]^{X\times Y}$, the set of metric relations from $X$ to $Y$  ordered pointwise, while $\VRel\RPos\reidx{f,g}$ is the precomposition with $f\times g$. 
The identity relation, the relational composition and the converse of a relation are defined as follows: 
\begin{align*}
\rid_X(x,x') &= \begin{cases}
0 & x = x' \\
\infty & x \ne x' 
\end{cases}
\\ 
(\relr\rcomp\rels)(x,z) &= \inf_{y\in\Car Y} (\relr(x,y) + \rels(y,z)) 
\\ 
\relr\rconv(y,x) &= \relr(x,y)
\end{align*}
where $\relr\in\VRel\RPos(X,Y)$ and $\rels\in\VRel\RPos(Y,Z)$. 

\item\label{ex:rel-doc:met} 
 Following Lawvere~\cite{Lawvere73}, a (symmetric) metric space $X$ consists of a set $\Car{X}$ and a function 
$\fun{\dist_X}{\Car{X}\times\Car{X}}{[0,\infty]}$ such that $\dist_X(x,x)=0$ and $\dist_{X}(x,x')=\dist_{X}(x',x)$ and $\dist_{X}(x,x')+\dist_{X}(x',x'')\ge\dist_{X}(x,x'')$. 
A metric space $X$ is \emph{separated} if additionally $\dist_X(x,x') = 0$ implies $x = x'$. 
Thus, separated metric spaces are essentially usual metric spaces where we also admit points with an infinite distance. 
Let $\Met$ be the category of separated metric spaces and non-expansive maps, i.e., 
arrows $\fun{f}{X}{Y}$ are functions $\fun{f}{\Car{X}}{\Car{Y}}$ such that 
$\dist_X(x,x') \ge \dist_Y(f(x),f(x'))$. 
A \emph{bimodule} $\phi$ between $X$ and $Y$ 
is a function 
from $\Car{X}\times\Car{Y}$ such that 
$\dist_{X}(x,x')+\phi(x,y)+\dist_{Y}(y,y')\ge\phi(x',y')$. 
We denote by $\BM(X,Y)$ the poset of bimodules between $X$ and $Y$ with the pointwise order. 
This assignment extends to a functor  $\fun{\BM}{\bop\Met}{\Pos}$ that acts on pairs of arrows by precomposition. 
The functor $\BM$ is a relational doctrine, where relational composition and the converse of a relation are defined as in \cref{ex:rel-doc:vrel} and where the identity $\rid_X$ is $\dist_X$.  
\end{enumerate}
\end{example}

\begin{remark} 
The relational doctrines presented in \cref{ex:rel-doc:vrel,ex:rel-doc:met} of \cref{ex:rel-doc} belong to a larger class of relational doctrines, 
obtained by replacing $\RPos$ with an arbitrary commutative quantale. 
All the results we will present for these specific case easily extend to this larger class. 
\end{remark}

Relational doctrines are the objects of the 2-category \RDtn. 

A 1-arrow \oneAr{F}{\RDoc}{\SDoc} from \fun{\RDoc}{\bop\CC}{\Pos} to \fun{\SDoc}{\bop\D}{\Pos}, 
is a pair \ple{\fn{F},\lift{F}} consisting of 
a functor \fun{\fn{F}}{\CC}{\D} and a natural transformation 
\nt{\lift{F}}{\RDoc}{\SDoc\circ \bop{\fn{F}}},

\[
\xymatrix@C=7.5em@R=1em{
{\bop\CC}\ar[rd]^(.4){\RDoc}_(.4){}="P"
\ar[dd]_{\bop{\fn{F}}}^{}="F"
&\\
 & {\ct{Pos}}\\
{\bop\D}\ar[ru]_(.4){\SDoc}^(.4){}="R"&
\ar"P";"R"_{\lift{F}\kern.5ex\cdot\kern-.5ex}="b"
}
\]
preserving relational identities, composition and converse, that is 
\begin{align*}
\rid_{\fn{F}X} &= \lift{F}_{X,X}(\rid_X)\\ 
\lift{F}_{X,Y}(\relr)\rcomp \lift{F}_{Y,Z}(\rels) &= \lift{F}_{X,Z}(\relr\rcomp\rels)\\ 
\lift{F}_{X,Y}(\relr))\rconv &= \lift{F}_{Y,X}(\relr\rconv)
\end{align*}
for $\relr\in\RDoc(X,Y)$ and $\rels\in\RDoc(Y,Z)$. 

A 2-arrow \twoAr{\theta}{F}{G} is a natural transformation \nt{\theta}{\fn{F}}{\fn{G}} such that 
$\lift{F}_{X,Y} \order \SDoc\reidx{\theta_X,\theta_Y}\circ \lift{G}_{X,Y}$, for all objects $X,Y$ in the base of $\RDoc$
\[
\xymatrix@C=18em@R=1.5em{
{\bop\CC}\ar[rd]^(.4){\RDoc}_(.4){}="P"
\ar@<-1ex>@/_/[dd]_{\bop{\fn{F}}}^{}="F"\ar@<1ex>@/^/[dd]^{\bop{\fn{F'}}}_{}="G"&\\
 & {\ct{Pos}}\\
{\bop\D}\ar[ru]_(.4){\SDoc}^(.4){}="R"&
\ar@/_/"P";"R"_{\lift{F}\kern.5ex\cdot\kern-.5ex}="b"
\ar@<1ex>@/^/"P";"R"^{\kern-.5ex\cdot\kern.5ex \lift{F'}}="c"
\ar"G";"F"_{.}^{\theta\op}\ar@{}"b";"c"|{\le}}
\]

We briefly recall some basic facts on relational doctrines and refer the reader to  \cite{DagninoP23,DagninoP25apal} for details. 
Fix a relational doctrine \fun{\RDoc}{\bop{\CC}}{\Pos}. 

For every arrow \fun{f}{X}{Y} in \CC the relation $\gr{f}=\RDoc\reidx{f,\id_Y}(\rid_Y)\in\RDoc(X,Y)$ is called \emph{graph} of $f$. 
Graphs commutes with composition and identities in the sense that  
$\gr{g\circ f} = \gr{f}\rcomp\gr{g}$ and $\gr{\id_X} = \rid_X$. 

An immediate use of graphs is in writing reindexing \fun{\RDoc\reidx{f,g}}{\RDoc(X,Y)}{\RDoc(A,B)} in relational terms and, moreover, in showing that each $\RDoc\reidx{f,g}$ has a left adjoint \fun{\Ex^\RDoc\reidx{f,g}}{\RDoc(A,B)}{\RDoc(X,Y)}. Indeed 
\[
\RDoc\reidx{f,g}(\relr) = \gr{f} \rcomp \relr \rcomp \gr{g}\rconv 
\qquad 
\Ex^\RDoc\reidx{f,g}(\rels) = \gr{f}\rconv \rcomp \rels \rcomp \gr{g} 
\]
The condition $\gr{f}=\gr{g}$ (or equivalently  $\rid_X\le \gr{f}\rcomp\gr{g}\rconv$ or also $\rid_X\le \gr{g}\rcomp\gr{f}\rconv$) defines a notion of equality between the arrows of $\CC$ which is, in general, weaker then the actual equality (see \cite{DagninoP25tac,DagninoP25apal} for a more detailed analysis). 
We say that  relational doctrine $\fun{\RDoc}{\bop\CC}{\Pos}$ is \emph{extensional} when for all \fun{f,g}{X}{Y} in \CC, if $\gr{f} = \gr{g}$, then $f = g$. 

\begin{example}\label[ex]{ex:rel-doc-ex}
All relational doctrines presented in \cref{ex:rel-doc} are extensional. In particular doctrines of the form $\Span\CC$ and $\jmSpan\CC$ in \refItem{ex:rel-doc}{span} and in \refItem{ex:rel-doc}{jmspans} are extensional as in both examples $\rid_X$ is $\spn{X}{\id_X}{X}{\id_X}{X}$; the relational doctrines $\VRel\RPos$ presented in \refItem{ex:rel-doc}{vrel} is extensional as for every set $X$ the metric relation $\rid_X$ is two-valued; finally,  the relational doctrines $\BM$ presented in \refItem{ex:rel-doc}{met} is extensional because metric spaces in $\Met$ are separated: 
if $\dist_Y(f(x),y) = \gr{f}(x,y) = \gr{g}(x,y) = \dist_Y(g(x),y)$, 
then $0=\dist_Y(f(x),f(x))=\dist_Y(g(x),f(x))$ so $f(x)=g(x)$.
\end{example}

A relation  $\relr\in\RDoc(X,Y)$ is
\begin{center}
\begin{tabular}{ll}
\emph{functional} if $\relr\rconv \rcomp \relr \order \rid_Y$&\emph{injective} if $\relr\rcomp\relr\rconv \order \rid_X$\\[1ex]
\emph{total} if $\rid_X \order \relr\rcomp\relr\rconv$&\emph{surjective} if $\rid_Y \order \relr\rconv \rcomp \relr$
\end{tabular}
\end{center}

For every $f$, the graph $\gr{f}$ is functional and total. The next proposition shows that 
 functional and total relations are discretely ordered.

\begin{proposition}\label[prop]{prop:fun-ord}
Let $\RDoc$ be a relational doctrine over \CC. 
For functional and total relations $\relr,\rels\in\RDoc(X,Y)$
if $\relr\order\rels$, then $\relr = \rels$. 
\end{proposition}
\begin{proof} 
If $\relr\order\rels$, then 
$\rels = \rid_X\rcomp\rels 
       \order \relr \rcomp \relr\rconv \rcomp \rels 
       \order \relr \rcomp \rels\rconv \rcomp \rels 
       \order \relr \rcomp \rid_Y 
       = \relr$.
\end{proof}

We will say that an arrow $f$ is 
\emph{$\RDoc$-injective} (resp. \emph{$\RDoc$-surjective}) or simply injective when $\RDoc$ is clear, when its graph $\gr{f}$ is injective (resp. surjective). A arrow is \emph{$\RDoc$-bijective} when it is  $\RDoc$-injective and $\RDoc$-surjective.

In general, monomorphisms (resp. epimorphisms) of $\CC$ and $\RDoc$-injections (resp. $\RDoc$-surjections) determines different classes of arrows, unless one assumes some extra structure such as in the following two propositions.

\begin{proposition}\label{prop:extomonic}In every extensional relational doctrine \fun{\RDoc}{\bop\CC}{\Pos} each $\RDoc$-injective arrow is monic, each $\RDoc$-surjective arrow is epic.
\end{proposition}
\begin{proof} Take a $\RDoc$-injective arrow $\fun{m}{A}{B}$ and let $\fun{f,g}{X}{B}$ be such that $mf=mg$. Then $\gr{mf}=\gr{mg}$ so $\rid_X\le \gr{f}\rcomp\gr{m}\rcomp\gr{m}\rconv\rcomp\gr{g}\rconv$. By $\RDoc$-injectivity of $m$ one has $\rid_X\le \gr{f}\rcomp\gr{g}\rconv$, then $f=g$ by extensionality. A similar argument proves that $\RDoc$-surjective arrow are epimorphisms.
\end{proof}

\begin{proposition}\label{prop:splitepi}
In every  relational doctrine \fun{\RDoc}{\bop\CC}{\Pos} retractions in $\CC$ are $\RDoc$-surjective and sections are $\RDoc$-injective.
\end{proposition}
\begin{proof} 
Let $\fun{e}{X}{Y}$ and $\fun{r}{Y}{X}$ be such that  $\id_Y=er$, then $\rid_Y=\gr{\id_Y}=\gr{er}=\gr{r}\rcomp\gr{e}$. So $\rid_Y= \rid_Y\rconv\rcomp\rid_Y=\gr{e}\rconv\gr{r}\rconv\rcomp\gr{r}\rcomp\gr{e}$ and since $\gr{r}$ is functional $\rid_Y\le \gr{e}\rconv\rcomp\gr{e}$. Similarly one proves the $\RDoc$-injectivity of $r$. 
\end{proof}

\subsection{Quotients and the extensional quotient completion}
Let $X$ be an object of \CC. 
An \emph{$\RDoc$-equivalence relation} on $X$ is a relation 
$\eqrelr$ in $\RDoc(X,X)$ satisfying
\begin{center}
\begin{tabular}{lll}
reflexivity: $\rid_X\order\eqrelr$& symmetry: $\eqrelr\rconv\order\eqrelr$&transitivity: $\eqrelr\rcomp\eqrelr\order\eqrelr$
\end{tabular}
\end{center}
Every arrow \fun{f}{X}{Y} in \CC induces a $\RDoc$-equivalence relation on $X$, dubbed \emph{kernel} of $f$, given by $\gr{f}\rcomp\gr{f}\rconv$.
One immediately sees that for the relational doctrine \fun{\Rel}{\bop\Set}{\Pos} of set-theoretic relations the kernel of a function $f$ is precisely the set $\{(x,x')\mid f(x)=f(x')\}$. 

A \emph{quotient arrow} of $\eqrelr$ is an arrow \fun{q}{X}{W} in \CC such that 
$\eqrelr \order \gr{q}\rcomp\gr{q}\rconv$ and, 
such that for every arrow \fun{f}{X}{Z} with $\eqrelr\order\gr{f}\rcomp\gr{f}\rconv$, there is a unique arrow \fun{h}{W}{Z} with $f = h\circ q$.  
We say that a quotient arrow \fun{q}{X}{W}  for $\eqrelr$ is 
\emph{effective} if $\eqrelr = \gr{q}\rcomp\gr{q}\rconv$, i.e., its kernel coincides with the equivalence relation, 
and it is \emph{descent} if $\gr{q}\rconv\rcomp\gr{q} = \rid_W$, i.e., $q$ is sujective. 
Finally, we say that \emph{$\RDoc$ has quotients} if every $\RDoc$-equivalence relation admits an effective descent quotient arrow.

\begin{example} \label[ex]{ex:rel-doc-q} 
\begin{enumerate}
\item\label{ex:rel-doc-q:span}
For the relational doctrine $\Span\CC$ of spans over a category \CC with weak pullbacks (see \refItem{ex:rel-doc}{span})
a $\Span\CC$-equivalence relation is a pseudo-equivalence relation of \CC in the sense of \cite{CarboniA:regec}.
\item\label{ex:rel-doc-q:jmspans}
For the relational doctrine  $\jmSpan\CC$ of jointly monic spans over a regular category $\CC$ (see \refItem{ex:rel-doc}{jmspans}) a $\jmSpan\CC$-equivalence relation  is an equivalence relation in \CC and quotient arrows are regular epimorphisms. 
Moreover  \CC is exact if and only if $\jmSpan\CC$ has quotients. 
\item\label{ex:rel-doc-q:vrel}
For the relational doctrine $\VRel\RPos$ of metric relations  (see  \refItem{ex:rel-doc}{vrel}) a $\VRel\RPos$-equivalence relation $d$ over a set $X$ is a function $d:X\times X\to [0,\infty]$ suchthat $d(x,x')=0$ (reflexivity), $d(x,x')=d(x',x)$ (symmetry) and $d(x,x')+d(x',x'')\geq d(x,x'')$ (transitivity); in other words, the $\VRel\RPos$-equivalence relations over a set $X$ are the distances on $X$.
\item\label{ex:rel-doc-q:met} 
Consider the relational doctrine  $\fun{\BM}{\bop\Met}{\Pos}$ of bimodules over separated metric spaces as in \refItem{ex:rel-doc}{met}. 
A $\BM$-equivalence relation over a metric space $X$ is a distance $\rho$ over $\Car{X}$ such that $\dist_{X}(x,x')\geq \rho(x,x')$. 
Note that these metrics are not required to be separated.  Therefore, in order to build a quotient, we have to force separation. 
To this end, given a $\BM$-equivalence relation $\rho$ over $X$, we consider the equivalence relation $\sim_\rho$ on the set $\Car{X}$ defined by $x\sim_\rho x'$ if and only if $\rho(x,x') = 0$ and the metric space 
$X_\rho$ where $\Car{X_\rho} = \Car{X}/\sim_\rho$ and $\dist_{X_\rho}(\eqc{x},\eqc{x'}) = \rho(x,x')$. 
It is easy to check that $X_\rho$ is a separated metric space. 
Moreover, the quotient projection function $\fun{\pi_\rho}{\Car{X}}{\Car{X}/\sim_\rho}$ is non-expansive and gives a quotient arrow $\fun{\pi_\rho}{X}{X_\rho}$ of $\rho$ in $\BM$. 
More generally, one can prove that an arrow in $\Met$ is a quotient arrow for a $\BM$-equivalence relation if and only if its underlying function is surjective. 
The underling function of $\pi_\rho$ is trivially surjective. 
Conversely, given an arow $\fun{f}{X}{Y}$ whose underlying function is surjective, 
denote by $k_f$ the kernel of $f$, that is, the $\BM$-equivalence relation on $X$ defined by 
$k_f(x,x') = \dist_Y(f(x),f(x'))$ 
and take the quotient projection $\fun{\pi_{k_f}}{X}{X_{k_f}}$. 
The function $\fun{j}{\Car{X_{k_f}}}{\Car{Y}}$, determined by $j(\eqc{x}) = f(x)$ is well-defined, since $Y$ is separated, and non-expansive. 
It is an isomorphism such that $j\circ \pi_{k_f}=f$, proving that $f$ is a quotient arrow. 
\end{enumerate}
\end{example}

Any 1-arrow \oneAr{F}{\RDoc}{\SDoc}  in \RDtn 
preserves equivalence relations, that is, 
if $\eqrelr$ is an $\RDoc$-equivalnece relation on $X$, then $\lift{F}_{X,X}(\eqrelr)$ is an $\SDoc$-equivalence relation on $\fn{F}X$. 
We say that a 1-arrow \oneAr{F}{\RDoc}{\SDoc} preserves quotients if, 
whenever $q$ is a quotient arrow for an $\RDoc$-equivalence relation $\eqrelr$ on $X$, the arrow $\fn{F}q$ is a quotient arrow for the $\SDoc$-equivalence relation $\lift{F}_{X,X}(\eqrelr)$ on $\fn{F}X$. 
We denote by $\QRDtn$ the 2-full 2-subcategory of $\RDtn$ whose objects are relational doctrines with quotients and whose 1-arrows are those of $\RDtn$ that preserves quotients. 

We now describe the construction presented in \cite{DagninoP23, DagninoP25apal}, that takes a relational doctrine 
\fun{\RDoc}{\bop\CC}{\Pos} and produces another relational doctrine 
\fun{\EQR\RDoc}{\bop{\EQC\RDoc}}{\Pos} which is extensional and has quotients. 
The category $\EQC\RDoc$ is defined as follows: 
\begin{itemize}
\item an object is a pair \ple{X,\eqrelr} of an object $X$ in \CC and an $\RDoc$-equivalence relation on it; 
\item an arrow \fun{\eqc{f}}{\ple{X,\eqrelr}}{\ple{Y,\eqrels}} is an equivalence class of arrows \fun{f}{X}{Y} in \CC such that $\eqrelr \order \gr{f}\rcomp\eqrels\rcomp\gr{f}\rconv$, modulo the equivalence reltion stating that two such arrow $f$ and $g$ are equivalent if and only if $\eqrelr\order\gr{f}\rcomp\eqrels\rcomp\gr{g}\rconv$. 
\end{itemize}
The functor $\EQR\RDoc$ is then defined as follows:
\[
\EQR\RDoc(\ple{X,\eqrelr},\ple{Y,\eqrels}) = \{ \relr\in\RDoc(X,Y) \mid \eqrelr\rcomp\relr\rcomp\eqrels\order\relr \}
\qquad 
\EQR\RDoc\reidx{\eqc{f},\eqc{g}} = \RDoc\reidx{f,g} 
\]
The elements of the fibre $\EQR\RDoc(\ple{X,\eqrelr},\ple{Y,\eqrels})$ are called \emph{descent data} for $\eqrelr$ and $\eqrels$. 
Finally, the relational operations of $\EQR\RDoc$ are the same as in $\RDoc$ except the identity which is given by $\rid_{\ple{X,\eqrelr}} = \eqrelr$. 

The relational doctrine $\EQR\RDoc$ has quotients: 
an $\EQR\RDoc$-equivalence relation $\eqrels$ on $\ple{X,\eqrelr}$ is  an $\RDoc$-equivalence relation on $X$ such that $\eqrelr\order\eqrels$ and the quotient arrow of $\eqrels$ is \fun{\eqc{\id_X}}{\ple{X,\eqrelr}}{\ple{X,\eqrels}}.

\begin{example} \label[ex]{ex:rel-doc-eqc} 
\begin{enumerate}
\item\label{ex:rel-doc-eqc:span}
If \CC has weak finite limits, 
then $\EQR{\Span\CC}$ is isomorphic to the doctrine of jointly monic spans over  $\CC_{ex/wlex}$, the exact completion of a weakly lex category as in \cite{CarboniA:regec} (crf. also \cref{rem:MillyPino}).  We did not find a reference for this claim whose proof is a bit lengthy and, as such, it is postponed to \cref{ss:spans}. It can be safetly skipped by those readers who are not interested in it.

\item\label{ex:rel-doc-eqc:vrel}
As proved in \cite{DagninoP23,DagninoP25apal}, the relational doctrine $\EQR{\VRel\RPos}$ is equivalent in $\EQRDtn$ to the relational doctrine $\fun{\BM}{\bop{\Met}}{\Pos}$ of bimodules over  separated metric spaces.
\end{enumerate}
\end{example}

\begin{remark}\label{rem:MillyPino} 
In \cite{DagninoP25apal} we showed that relational doctrines that are cartesian (i.e the 1-arrows $\fun{!_\RDoc}{\RDoc}{1}$ and $\fun{\Delta_\RDoc}{\RDoc}{\RDoc\times \RDoc}$
have a right adjoint in \RDtn) and modular (i.e. relations satisfy $\alpha\rcomp\beta \wedge \gamma \le \alpha\rcomp(\beta\wedge \alpha\rconv\rcomp\gamma)$) are equivalent to elementary and existential doctrines as in \cite{MaiettiME:quofcm}. 
If $\RDoc$ is cartesian and modular, so is $\EQR{\RDoc}$ and, in this case, $\EQR{\RDoc}$ coincides with the completion introduced by Maietti and Rosolini in \cite{MaiettiME:eleqc,MaiettiME:quofcm}, dubbed elementary quotient completion. So in particular the examples listed in \cite{MaiettiME:eleqc,MaiettiME:quofcm} such as the $ex/wlex$ completion of a category $\CC$ with strong products, and other non-exact categories of equivalence relations, such as the category of equilogical spaces, the category of assemblies over a partial combinatory algebra, the categories of $H$-valued partial equivalence relations for a fixed locale $H$ and categories of setoids over type theories are examples of relational quotient completion for an appropriately chosen relational doctrine.
\end{remark}

We conclude the section recalling that the construction of $\EQR\RDoc$ is universal in the sense of the following proposition, whose proof can be found in \cite{DagninoP25apal}. Let $\EQRDtn$ be the full 2-subcategory of $\QRDtn$ on those relational doctrines with quotients which are also extensional, denote by $\fun{\EQRFun}{\EQRDtn}{\RDtn}$ the forgetful functor and observe that the relational quotient completion of a relational doctrine, i.e. the construction $\RDoc\mapsto\EQR\RDoc$, extends to a 2-functor $\fun{\REQFun}{\RDtn}{\EQRDtn}$.

\begin{proposition}\label{prop:eqc-mnd}The relational quotient completion determines a biadjunction $\REQFun\dashv\EQRFun$. 
\end{proposition}

\subsection{Quotient completion of doctrines of spans}\label{ss:spans}

Here we prove the claim in \refItem{ex:rel-doc-eqc}{span} stating that, for a category $\CC$ with weak finite limits, $\EQR{\Span\CC}$ is isomorphic in $\EQRDtn$ to $\fun{\jmSpan{\CC_{ex/wlex}}}{\bop{(\CC_{ex/wlex})}}{\Pos}$. 
As already mentioned, we are including this full proof for sake of completeness, as we have not found it in the literature. 
However, it is not essential for understanding the rest of the paper.

First of all, let us notice that 
the base of $\EQR{\Span\CC}$ is precisely $\CC_{ex/wlex}$. 
To show that also the doctrines are isomorphic we  need some instrumental lemmas.

\begin{lemma}\label{l:isofrecce}
Suppose $\RDoc$ and $\SDoc$ are relational doctrines over the same base $\CC$ and \oneAr{F}{\RDoc}{\SDoc} is a 1-arrow where $\fun{\fn{F}}{\CC}{\CC}$ is the identity on $\CC$. If for every $X,Y$ in $\CC$ the map $\fun{\lift{F}_{X,Y}}{\RDoc(X,Y)}{\SDoc(X,Y)}$ is an isomorphism in $\Pos$, then $\RDoc$ and $\SDoc$ are isomorphic.
\end{lemma}
\begin{proof}For every $X,Y$, call $G_{X,Y}$ the inverse of $\lift{F}_{X,Y}$ in $\Pos$.  The following equations 
\begin{align*} 
G_{X,X}(\rid_X)&=G_{X,X}(\lift{F}_{X,X}(\rid_X))=\rid_X\\
G_{X,Z}(\relr\rcomp\rels)&=G_{X,Z}(\lift{F}_{X,Y}G_{X,Y}(\relr)\rcomp \lift{F}_{Y,Z}G_{Y,Z}(\rels))\\
&=G_{X,Z}\lift{F}_{X,Z}(G_{X,Y}(\relr)\rcomp G_{Y,Z}(\rels))=G_{X,Y}(\relr)\rcomp G_{Y,Z}(\rels)\\
G_{X,Y}(\relr\rconv)&= G_{X,Y}((\lift{F}_{X,Y}G_{X,Y}(\relr))\rconv)=G_{X,Y}\lift{F}_{X,Y}((G_{X,Y}(\relr))\rconv)=G_{X,Y}(\relr)\rconv
\end{align*} 
 show that the inverses commute with the relational operations, as such they determine a 1-arrow of relational doctrines $\fun{}{\SDoc}{\RDoc}$.
\end{proof}

Let $\ct{E}$ be an exact category and $\ct{P}$ a regular projective cover of it. Denote by $\Span{\ct{E}}_{\ct{P}}(X,Y)$ the subposet of $\Span{\ct{E}}(X,Y)$ on those spans whose vertex is in $\ct{P}$. If $\relr=\spn{X}{f}{A}{g}{Y}$ is a span  in $\Span{\ct{E}}(X,Y)$ and $\fun{q}{P}{A}$ and $\fun{q'}{P'}{A}$ are regular epimorphisms, then,  $q$ and $q'$ factors through each other, so the spans $\spn{X}{fq}{P}{gq}{Y}$ and $\spn{X}{fq'}{P'}{gq'}{Y}$ represent the same element in $\Span{\ct{E}}_{\ct{P}}(X,Y)$. 
We denote it by $\lift{\coveredarr}(\relr)$.

Note also that a regular projective cover of a pullback in $\ct{E}$ determines a weak pullback, i.e.  given $\fun{e}{X}{Y}$ and $\fun{f}{A}{Y}$ the squares $f\circ (f^*e\circ q)=e\circ (e^*f\circ q)$, where $q$ is a regular epimorphism with a regular projective domain, is a weak pullback. 
As a consequence of this and since regular epimorphisms compose, for every $f$ and $e$ as above there is a weak pullback of $e$ along $f$ 
which is a regular epimorphism whose domain is in $\ct{P}$.
This suffices to show that (equivalence classses of) spans with a regular projective vertex form a relational doctrine $\fun{\Span{\ct{E}}_{\ct{P}}}{\bop{\ct{E}}}{\Pos}$ where the composition of $\relr$ and $\rels$ is $\lift{\coveredarr}(\relr\rcomp\rels)$ and the identity relation over $A$ is $\lift{\coveredarr}(\rid_A)$ and $(\Span{\ct{E}}_{\ct{P}})_{f,g}= \lift{\coveredarr}(\Span{\ct{E}}_{f,g})$. 
Moreover the way we defined the relational operations of $\Span{\ct{E}}_{\ct{P}}$ ensures that \oneAr{\coveredarr}{\jmSpan{\ct{E}}}{\Span{\ct{E}}_{\ct{P}}} is indeed a 1-arrow as depicted below 
\[
\xymatrix@C=7.5em@R=1em{
{\bop{\ct{E}}}\ar[rd]^(.4){\jmSpan{\ct{E}}}_(.4){}="P"
\ar[dd]_{\bop{\fn{\coveredarr}}}^{}="F"
&\\
 & {\ct{Pos}}\\
{\bop{\ct{E}}}\ar[ru]_(.4){\Span{\ct{E}}_{\ct{P}}}^(.4){}="R"&
\ar"P";"R"_{\lift{\coveredarr}\kern.5ex\cdot\kern-.5ex}="b"
}
\]
 where \fun{\fn{\coveredarr}}{\ct{E}}{\ct{E}} is the identity functor on $\ct{E}$.

The following lemma is a stronger version of Lemma 35 in \cite{CarboniA:regec}. 

\begin{lemma} \label{l:analogo0} 
The 1-arrow \oneAr{\coveredarr}{\jmSpan{\ct{E}}}{\Span{\ct{E}}_{\ct{P}}} 
is an isomorphism.
\end{lemma}
\begin{proof} 
After \cref{l:isofrecce}, we only need to show that $\fun{\lift{\coveredarr}_{X,Y}}{\jmSpan{\ct{E}}(X,Y)}{\Span{\ct{E}}_{\ct{P}}(X,Y)}$ has an inverse in $\Pos$. This is given noticing that in an exact category every span factors through a jointly monic one, where the mediating arrow is a regular epimorphism.
\end{proof}

Let \CC be a category with weak finite limits and consider the relational doctrine $\fun{\Span\CC}{\bop\CC}{\Pos}$. 
Denote by  $\ct{C}_\Delta$ the regular projective cover of $\CC_{ex/wlex}$ determined by the objects of the form $(A,\rid_A)$. The following is an obvious corollary of \cref{l:analogo0} 

\begin{corollary}\label{c:analogo}
The doctrines $\Span{\ct{C_{ex/wlex}}}_{\ct{C}_\Delta}$ and $\jmSpan{\ct{C_{ex/wlex}}}$ are isomorphic.
\end{corollary}

The final stage of the proof consists in showing that $\Span{\ct{C_{ex/wlex}}}_{\ct{C}_\Delta}$ and $\EQR{\Span\CC}$ are isomorphic.
There is a morphism of doctrines \oneAr{i}{\EQR{\Span\CC}}{\Span{\ct{C_{ex/wlex}}}_{\ct{C}_\Delta}} where $\fn{i}$ is the identity functor on $\ct{C_{ex/wlex}}$ and the component of $\lift{i}$ at $\ple{X,\eqrelr}$ and $\ple{Y,\eqrels}$ sends a span $\spn{X}{f}{A}{g}{Y}$ to $\spn{\ple{X,\eqrelr}}{[f]}{\ple{A,\rid_A}}{[g]}{\ple{Y,\eqrels}}$. 

To build the inverse of $i$ the following lemma will be useful.

\begin{lemma} \label{l:chiudo} 
Let $\RDoc$ be a relational doctrine over $\CC$. For every equivalence relations $\eqrelr, \eqrels$ over $X$ and $Y$, respectively, the assignment $\relr\mapsto \eqrelr\rcomp\relr\rcomp\eqrels$ determines a monotone function $\fun{}{\RDoc(X,Y)}{\EQR\RDoc(\ple{X,\eqrelr},\ple{Y,\eqrels})}$.
\end{lemma}
\begin{proof} 
Monotonicity follows from monotonicity of relational composition. Recall that $\EQR\RDoc(\ple{X,\eqrelr},\ple{Y,\eqrels})= \{ \relr\in\RDoc(X,Y) \mid \eqrelr\rcomp\relr\rcomp\eqrels\order\relr \}$, thus $\eqrelr\rcomp\relr\rcomp\eqrels$ belongs to it by transitivity of $\eqrelr$ and $\eqrels$.
\end{proof}

Observe that an element  in $\Span{\CC}(X,Y)$, represented by $\spn{X}{f}{A}{g}{Y}$, can be written as $\gr{f}\rconv\rcomp\gr{g}$ where the relational operations are those of $\Span{\CC}$ and thus $\gr{f}$ is represented by the span 
$\spn{A}{\id_A}{A}{f}{X}$ and similarly for $\gr{g}$. 
Moreover, any span $\spn{\ple{X,\eqrelr}}{[f]}{\ple{A,\rid_A}}{[g]}{\ple{Y,\eqrels}}$ 
in $\ct{C_{ex/wlex}}$ is composed by legs, whose representative $f$ and $g$ are arrows of $\CC$ such that $\rid_A\order\gr{f}\rcomp\eqrelr\rcomp\gr{f}\rconv$ and $\rid_A\order\gr{g}\rcomp\eqrels\rcomp\gr{g}\rconv$ in $\Span{\CC}(A,A)$.

Thus, a candidate to be the inverse \oneAr{i'}{\Span{\ct{C_{ex/wlex}}}_{\ct{C}_\Delta}}{\EQR{\Span\CC}} 
of $i$ can be find choosing, for each $\spn{\ple{X,\eqrelr}}{[f]}{\ple{A,\rid_A}}{[g]}{\ple{Y,\eqrels}}$ in $\Span{\ct{C_{ex/wlex}}}_{\ct{C}_\Delta}(\ple{X,\eqrelr},\ple{Y,\eqrels})$, a span $\spn{X}{f}{A}{g}{Y}$ and use \cref{l:chiudo} to map $\gr{f}\rconv\rcomp\gr{g}$ in 
$\Span\CC(X,Y)$ to $\eqrelr\rcomp\gr{f}\rconv\rcomp\gr{g}\rcomp\eqrels$ in $\EQR{\Span{\CC}}$
$(\ple{X,\eqrelr},\ple{Y,\eqrels})$. This assignment does not depend on the choice of the representative: if  $\spn{X}{f'}{A}{g'}{Y}$ is another choice, then $\rid_A\order\gr{f}\rcomp\eqrelr\rcomp\gr{f'}\rconv$ and $\rid_A\order\gr{g'}\rcomp\eqrels\rcomp\gr{g}\rconv$ in $\Span{\CC}(A,A)$; equivalently, $\gr{f}\rconv\order\eqrelr\rcomp\gr{f'}\rconv$ and $\gr{g}\order\gr{g'}\rcomp\eqrels$, so
\[
\eqrelr\rcomp\gr{f}\rconv\rcomp\gr{g}\rcomp\eqrels
\order
\eqrelr\rcomp\eqrelr\rcomp\gr{f'}\rconv\rcomp\gr{g'}\rcomp\eqrels\rcomp\eqrels
=
\eqrelr\rcomp\gr{f'}\rconv\rcomp\gr{g'}\rcomp\eqrels
\]
analogously one shows that $\eqrelr\rcomp\gr{f'}\rconv\rcomp\gr{g'}\rcomp\eqrels\order\eqrelr\rcomp\gr{f}\rconv\rcomp\gr{g}\rcomp\eqrels$. This proves that $i'$ is well defined function.

 It is also monotone: take a bigger span $\spn{\ple{X,\eqrelr}}{[h]}{\ple{B,\rid_B}}{[l]}{\ple{Y,\eqrels}}$, i.e. suppose that there is $\fun{k}{A}{B}$ with $[hk]=[f]$ and $[lk]=[g]$. In terms of the relational operations of $\Span{\CC}$ it is $\gr{f}\rconv\order \eqrelr\rcomp\gr{h}\rcomp\gr{k}\rconv$
 and $\gr{g}\order \gr{k}\rcomp\gr{l}\rcomp\eqrels$. So
 \[
 \eqrelr\rcomp\gr{f}\rconv\rcomp\gr{g}\rcomp\eqrels
 \order
 \eqrelr\rcomp\eqrelr\rcomp\gr{h}\rcomp\gr{k}\rconv\rcomp\gr{k}\rcomp\gr{l}\rcomp\eqrels\rcomp\eqrels
 \order
 \eqrelr\rcomp\gr{h}\rcomp\gr{l}\rcomp\eqrels
 \]
where we use transitivity of $\eqrelr$ and $\eqrels$ and functionality of $\gr{k}$. 
\cref{l:isofrecce} tell us that if we show that $i'$ is the inverse of $i$ as a monotone function, then $i'$ commutes with relational operations, providing the inverse of $i$ in $\RDtn$.

Since elements of $\EQR{\Span\CC}(\ple{X,\eqrelr},\ple{Y,\eqrels})$ are descent data, the composition $i'i$ is the identity. Moreover, by reflexivity of $\eqrelr$ and $\eqrels$, it holds that  $\relr\order \eqrelr\rcomp\relr\rcomp\eqrels$, so $\id\order ii'$. It remains to show that $ii'\order\id$.
 Recall that $\eqrelr\rcomp\gr{f}\rconv\rcomp\gr{g}\rcomp\eqrels$ is given by the two external legs of the following diagram of \CC

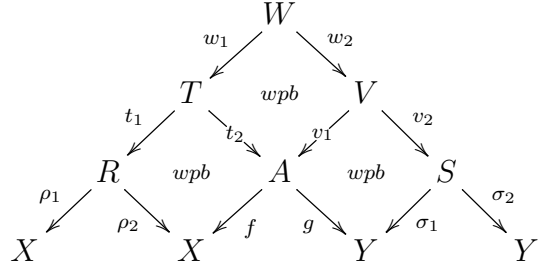
\[
\vcenter{\xymatrix@C=3ex@R=3ex{
&&& W \ar[ld]_-{w_1} \ar[rd]^-{w_2} \ar@{}[dd]|{wpb} &&& \\
&& T\ar[ld]_-{t_1} \ar[rd]|{t_2} \ar@{}[dd]|{wpb} && V\ar[ld]|{v_1} \ar[rd]^-{v_2} \ar@{}[dd]|{wpb}&& \\  
&R\ar[ld]_-{\eqrelr_1}\ar[rd]_-{\eqrelr_2} && A \ar[ld]^-{f} \ar[rd]_-{g} && S \ar[ld]^-{\eqrels_1}\ar[rd]^-{\eqrels_2}& \\ 
X&& X && Y &&Y
}}
\]
To show that $i(\eqrelr\rcomp\gr{f}\rconv\rcomp\gr{g}\rcomp\eqrels)\order i(\gr{f}\rconv\rcomp\gr{g})$ we need to find $\fun{k}{W}{A}$ such that $[\eqrelr_1t_1w_1]=[fk]$ and $[gk]=[\eqrels_2v_2w_2]$ in $\CC_{ex/wlex}$. That is, we need to find an arrow $\fun{r}{W}{R}$   with $\eqrelr_1r=\eqrelr_1t_1w_1$ and $\eqrelr_2r=fk$
and an arrow  $\fun{s}{W}{S}$ with 
$\eqrels_1s=gk$ and $\eqrels_2s=\eqrels_2v_2w_2$. It is immediate to verify that 
\[
\xymatrix{W\ar[r]^-{v_1w_2}&A}\quad \xymatrix{W\ar[r]^-{t_1w_1}&R}
\quad \xymatrix{W\ar[r]^-{v_2w_2}&S}
\]
are the desired $k$,
 $r$ and $s$, respectively.

\section{Projective objects} 
\label{sect:proj} 

In this section we provide a characterization of the essential image of the quotient completion presented in the \cref{sect:prelim}.
This  generalizes (and is inspired by) an analogous characterization for the exact completion of a category with (weak) finite limits in \cite{CarboniA:freecl,CarboniA:regec} and another one for the elementary quotient completion of existential elementary doctrines in \cite{MaiettiPR24}. 
To achieve this, we have to introduce the notion of \emph{projective object} with respect to quotient arrows in a relational doctrine. 

\begin{definition}\label[def]{def:proj-obj}
Let \fun{\RDoc}{\bop\CC}{\Pos} be a relational doctrine with quotients. 
An object $P$ in \CC is \emph{$\RDoc$-projective} (with respect to quotient arrows)  if 
for every arrow \fun{f}{P}{Y} and every quotient arrow \fun{q}{X}{Y} there is  \fun{h}{P}{X} such that $f = q\circ h$, as in the following diagram 
\[\xymatrix{
& X \ar[d]^-{q} 
\\
  P \ar@{..>}[ru]^-{h}
    \ar[r]_-{f} 
& Y 
}\]
\end{definition}

Note that this defines $\RDoc$-projective objects as those that are projective with respect to the class of (effective descent) quotient arrows.

\begin{example}\label[ex]{ex:proj-obj}
\begin{enumerate}
\item\label{ex:proj-obj:jspan}
Consider the doctrine $\jmSpan\CC$ of jointly monic spans over an exact category \CC 
(\cf \refItem{ex:rel-doc}{jmspans}). 
We have that 
an object $P$ is $\jmSpan\CC$-projective if and only if it is regular projective.
\item\label{ex:proj-obj:met} 
Consider the doctrine $\BM$ of bimodules over metric spaces 
(\cf \refItem{ex:rel-doc}{met}). 
A metric space $X$ is said to be trivial if $\dist_X(x,x') = \infty$ whenever $x\ne x'$. 
Trivial metric spaces are $\BM$-projectives.
\end{enumerate}
\end{example}

 An important property of $\RDoc$-projective objects is that they are closed under retracts. 

\begin{proposition}\label{prop:proj-retract}
Let $\fun{\RDoc}{\bop\CC}{\Pos}$ be a relational doctrine and $\fun{r}{X}{Y}$ a retraction in \CC. 
If $X$ is $\RDoc$-projective, $Y$ is $\RDoc$-projective. 
\end{proposition}
\begin{proof}
Consider an arrow $\fun{f}{Y}{W}$ and a quotient arrow $\fun{q}{Z}{W}$. 
Since $X$ is $\RDoc$-projective, we have an arrow 
$\fun{g}{X}{Z}$ such that 
$q\circ g = f \circ r$. 
Since $r$ is a retraction, it has a section $\fun{s}{Y}{X}$, then we derive 
$q \circ (g\circ s) = f$, proving that $Y$ is $\RDoc$-projective. 
\end{proof}

Thanks to this proposition, we can show that \refItem{ex:proj-obj}{met} is an instance of a general fact about the extensional quotient completion. 

\begin{proposition}\label[cor]{cor:proj-obj-eqc}
Let $\RDoc$ be a relational doctrine and $\EQR\RDoc$ its 
extensional  quotient completion. 
Then, an object \ple{X,\eqrelr} is $\EQR\RDoc$-projective if and only if it is a retract of $\ple{X,\rid_X}$. 
\end{proposition} 
\begin{proof}
To prove the left-to-right implication, 
assume that \ple{X,\eqrelr} is $\EQR\RDoc$-projective and consider the quotient arrow $\fun{\eqc{\id_X}}{\ple{X,\rid_X}}{\ple{X,\eqrelr}}$. 
Then, we get an arrow 
$\fun{\eqc{f}}{\ple{X,\eqrelr}}{\ple{X,\rid_X}}$ 
such that the diagram  
\[\xymatrix{
& \ple{X,\rid_X} \ar[d]^-{\eqc{\id_X}} \\ 
  \ple{X\eqrelr}
  \ar[r]_-{\id_{\ple{X,\eqrelr}}}
  \ar[ur]^-{\eqc{f}} 
& \ple{X,\eqrelr}
}\]
commutes. This proves that $\ple{X,\eqrelr}$ is a retract of $\ple{X,\rid_X}$. 

To prove the converse, 
thanks to \cref{prop:proj-retract}, it suffices to show that $\ple{X,\rid_X}$ is $\EQR\RDoc$-projective. 
Consider an arrow 
$\fun{\eqc{f}}{\ple{X,\rid_X}}{\ple{Y,\eqrelr}}$ and a quotient arrow 
$\fun{\eqc{q}}{\ple{Z,\eqrels'}}{\ple{Y,\eqrels}}$. 
Without loss of generality, we can assume $Z = Y$ and $\eqc{q} = \eqc{\id_Y}$ as all quotient arrows are of this form up to isomorphism. 
Therefore, to conclude, we only have to check that $\eqc{f}$ is a well-defined arrow from $\ple{X,\rid_X}$ to $\ple{Y,\eqrels'}$, that is, 
$\rid_X \order \gr{f}\rcomp\eqrels'\rcomp\gr{f}\rconv$.
This is immediate as 
$\rid_X \order\gr{f}\rcomp\rid_Y'\rcomp\gr{f}\rconv \order \gr{f}\rcomp\eqrels'\rcomp\gr{f}\rconv$, since $\eqrels'$ is reflexive. 
\end{proof}

Another useful property of projective objects is that left adjoints preserve them.  

\begin{proposition}\label[prop]{prop:proj-obj-ladj}
Let \oneAr{F}{\RDoc}{\SDoc} be a 1-arrow in \QRDtn such that $\fn{F}$ has a left adjoint $L$. 
If $P$ is $\SDoc$-projective, $LP$ is $\RDoc$-projective. 
\end{proposition}
\begin{proof}
Consider an arrow \fun{f}{LP}{Y} and a quotient arrow \fun{q}{X}{Y} in the base of $\RDoc$. 
Since $F$ preserves quotients, we have that the arrow \fun{\fn{F}q}{\fn{F}X}{\fn{F}Y} is a quotient arrow in the base of $\SDoc$. 
Then, since $P$ is $\SDoc$-projective, we get an arrow \fun{h}{P}{\fn{F}X} such that the following diagram commutes
\[\xymatrix{
& \fn{F}X \ar[d]^-{\fn{F}q} 
\\
  P \ar[ru]^-{h} \ar[r]_-{f^\sharp} 
& \fn{F}Y 
}\]
where $f^\sharp$ is the transpose of $f$ along the adjunction $L \dashv\fn{F}$. 
Hence, transposing the above diagram we get the thesis. 
\end{proof}

For a relational doctrine $\RDoc$ with quotients, 
we denote by $\RProj\RDoc$ its restriction to the full subcategory $\Proj\RDoc$ of its base spanned by $\RDoc$-projective objects.  
The next definition  provides the central property for our characterization of the extensional quotient completion. 

\begin{definition}\label[def]{def:enough-proj}
Let \fun{\RDoc}{\bop\CC}{\Pos} be a relational doctrine with quotients. 
A full subcategory $\ct{G}$ of $\Proj\RDoc$ is said to be a \emph{$\RDoc$-projective cover} 
if, for every object $X$ in \CC, there is a quotient arrow \fun{q}{P}{X} where $P$ is in $\ct{G}$. 
We say that $\RDoc$ \emph{has enough projectives} if it has a $\RDoc$-projective cover. 
\end{definition}

In other words, a relational doctrine $\RDoc$ has enough projectives when every object can be presented as a quotient of a $\RDoc$-projective object. 
In a sense, if $\ct{G}$ is a $\RDoc$-projective cover, we know that objects of $\ct{G}$, that are $\RDoc$-projective, suffices to ``generate'' all objects in the base of $\RDoc$.
This idea is made precise in the next theorem and the subsequent corollary. 
In the following, 
given a 1-arrow \oneAr{F}{\SDoc}{\RDoc} in \RDtn into an extensional relational doctrine with quotients, we denote by \oneAr{F^\sharp}{\EQR\SDoc}{\RDoc} its transpose along the biadjunction $\REQFun\dashv\EQRFun$, that is, 
$F^\sharp = Q^\RDoc\circ\EQR{F}$, where $Q^\RDoc$ is the component of the counit at $\RDoc$. 
We need also to introduce a class of morphisms of doctrines that generalise the full and faithful functors between categories. To this purpose, observe that a full and faithful functor $F:\CC\to\ct{D}$ is a local isomorphism in the sense that for every $A$ and $B$ the mapping $\fun{F(-)}{\CC(A,B)}{\ct{D}(FA,FB)}$ is a bijection. If $\CC$ is the base of a relational doctrine $\RDoc$ one has at least two notions of morphism between objects $A$ and $B$, the arrows of $\CC$ and the elements of $\RDoc(A,B)$. We say that a 1-arrow $F$ is \emph{fully faithful} if $\fn{F}$ is fully faithful and $\lift{F}$ is a family of isomorphisms.

\begin{theorem}\label[thm]{thm:proj-obj}
Let \fun{\RDoc}{\bop\CC}{\Pos}  be an extensional relational doctrine with quotients and 
\oneAr{F}{\SDoc}{\RDoc} a fully-faithful  1-arrow in \RDtn. 
Then, the following are equivalent:
\begin{enumerate}
\item\label{thm:proj-obj:1} the 1-arrow \oneAr{F^\sharp}{\EQR\SDoc}{\RDoc} is an equivalence in \EQRDtn;
\item\label{thm:proj-obj:2} $F$ factors through $\RProj\RDoc$ and 
for every object $X$ in \CC, there is an object $P_X$ in the base of $\SDoc$ and a quotient arrow \fun{q_X}{\fn{F}P_X}{X} in \CC.
\end{enumerate}
\end{theorem}
\begin{proof}
We first prove the implication from \cref{thm:proj-obj:1} to \cref{thm:proj-obj:2}. 
Let \oneAr{G}{\RDoc}{\EQR\SDoc} be the pseudoinverse of $F^\sharp$ in \EQRDtn, 
hence, by hypothesis, both $F^\sharp$ and $G$ preserve quotients. 
First we show that, for every object $A$ in the base of $\SDoc$, the object $\fn{F}A$ is $\RDoc$-projective, which immediately implies that $F$ factors through $\RProj\RDoc$. 
Consider a quotient arrow \fun{q}{X}{Y} and an arrow \fun{f}{\fn{F}A}{Y} in \CC.
Since $\fn{F}A \cong \fn{F^\sharp}\ple{A,\rid_A}$ and $G$ is a pseudoinverse of $F^\sharp$, we deduce that 
$\fn{G}\fn{F}A \cong \ple{A,\rid_A}$, hence by \cref{prop:proj-obj-ladj} we get that $\fn{G}\fn{F}A$ is $\EQR\SDoc$-projective. 
Since $G$ preserves quotients, we know that 
\fun{\fn{G}q}{\fn{G}X}{\fn{G}Y} is a quotient arrow, hence by projectivity of $\fn{G}\fn{F}A$, we get an arrow
as in the next commutative diagram:
\[\xymatrix{
& \fn{G}X \ar[d]^-{\fn{G}q}  \\
  \fn{G}\fn{F}A \ar[r]_-{\fn{G}f} \ar[ur]^-{\eqc{h}} 
& \fn{G}Y 
}\]
Finally, since $\fn{G}$ is fully-faithful, there is an arrow \fun{h}{\fn{F}A}{X} such that 
$f = q\circ h$, as needed. 
 
Consider now an object $X$ in \CC  and suppose that $\fn{G}X = \ple{P_X,\eqrelr_X}$. 
Let us consider the object $P = \fn{F^\sharp}\ple{P_X,\rid_{P_X}}$ in \CC. 
The arrow  \fun{\eqc{\id_{P_X}}}{\ple{P_X,\rid_{P_X}}}{\ple{P_X,\eqrelr_X}} is a quotient arrow for the $\SDoc$-equivalence relation $\eqrelr_X$,  
thus the arrow \fun{\fn{F^\sharp}\id_{P_X}}{P}{\fn{F^\sharp}\ple{P_X,\eqrelr_X}} is a quotient arrow in $\RDoc$, because $F^\sharp$ preserves quotients. 
Therefore, because $\fn{F^\sharp}\ple{P_X,\eqrelr_X} = \fn{F^\sharp}\fn{G} X$ is isomorphic to $X$, as $F^\sharp$ is a pseudoinverse of $G$, and $\fn{F^\sharp}P_X = Q^\RDoc\ple{\fn{F}P_X,\rid_{\fn{F}P_X}}$ is isomorphic to $P$, 
we get the thesis. 

\medskip 

Towards a proof of the other direction, 
we build a pseudoinverse of $F^\sharp$. 
By hypothesis, 
for every object $X$ in \CC, we have an object $P_X$ in the base of $\SDoc$ and 
a quotient arrow \fun{p_X}{\fn{F}P_X}{X}, where $\fn{F}P_X$ is $\RDoc$-projective.
Moreover, since $\lift{F}$ is an isomorphism, 
we let $\eqrelr_X$ be the unique $\SDoc$-equivalence relation on $P_X$ such that 
$\lift{F}_{P_X,P_X}(\eqrelr_X) = \gr{p_X}\rcomp\gr{p_X}\rconv$. 
Consider now an arrow \fun{f}{X}{Y} in \CC. 
Since $\fn{F}P_X$ is $\RDoc$-projective and $\fn{F}$ is full, we get an arrow \fun{\hat{f}}{P:X}{P_Y} such that 
$p_Y\circ\fn{F}\hat{f} = f\circ p_X$. 
Moreover, we have that 
\begin{align*} 
\lift{F}_{P_X,P_X}(\eqrelr_X) 
  & = \gr{p_X}\rcomp\gr{p_X}\rconv 
    \order \gr{p_X}\rcomp\gr{f}\rcomp\gr{f}\rconv\rcomp\gr{p_X}\rconv 
    = \gr{\fn{F}\hat f}\rcomp\gr{p_Y}\rcomp\gr{p_Y}\rconv\rcomp\gr{\fn{F}\hat f}\rconv 
\\
  & = \lift{F}_{P_X,P_X}(\gr{\hat f}\rcomp\eqrelr_Y\rcomp\gr{\hat f}\rconv) 
\end{align*} 
Since $\lift{F}$ is an isomorphism, we deduce that 
$\eqrelr_X\order\gr{\hat f}\rcomp\eqrelr_Y\rcomp\gr{\hat f}\rconv$ holds, 
thus proving that \fun{\eqc{\hat f}}{\ple{P_X,\eqrelr_X}}{\ple{P_Y,\eqrelr_Y}} is a well defined arrow in the base of $\EQR\SDoc$. 
Furthermore, if $g,h$ are both arrows such that $p_Y\circ \fn{F}g = f\circ p_X$ and $p_Y\circ \fn{F}h = f\circ p_X$,  we have that 

\[
\lift{F}_{P_X,P_Y}(\gr{g}\rcomp\eqrelr_Y) 
   = \gr{\fn{F}g}\rcomp\gr{p_Y}\rcomp\gr{p_Y}\rconv 
    = \gr{\fn{F}h}\rcomp\gr{p_Y}\rcomp\gr{p_Y}\rconv 
 = \lift{F}_{P_X,P_Y}(\gr{h}\rcomp\eqrelr_Y)
\]

Again, since $\lift{F}$ is an isomorphism, we deduce that 
$\gr{g}\rcomp\eqrelr_Y = \gr{h}\rcomp\eqrelr_Y$, 
thus showing that $\eqc{g} = \eqc{h}$. 
Therefore, the arrow $\eqc{\hat f}$ is uniquely determined and so  the assignments 
$\fn{G}X = \ple{P_X,\eqrelr_X}$ and $\fn{G}f = \eqc{\hat f}$ determine a functor from \CC to the base of $\EQR\SDoc$. 
Further, given a relation $\relr\in\RDoc(X,Y)$, we denote by $\hat\relr$ the unique relation in $\SDoc(P_X,P_Y)$ such that $\lift{F}_{P_X,P_Y}(\hat\relr) = \gr{p_X}\rcomp\relr\rcomp\gr{p_Y}\rconv$, which again exists because $\lift{F}$ is an isomorphism. 
Then, we set $\lift{G}_{X,Y}(\relr) = \hat\relr$ and 
it is easy to check that \oneAr{G}{\RDoc}{\EQR\SDoc} is a well-defined 1-arrow in \RDtn. 
 
We check that $G$ preserves quotients. 
Let \fun{q}{X}{W} be a quotient arrow in $\RDoc$ for $\eqrelr$ and consider an arow \fun{\eqc{f}}{\ple{P_X,\eqrelr_X}}{\ple{Y,\eqrels}} such that 
$\hat\eqrelr \order \gr{f}\rcomp\eqrels\rcomp\gr{f}\rconv$. 
Then, we have the following diagram 
\[\xymatrix{
  \fn{F}P_X 
  \ar[rrrd]^-{\fn{F}f} 
  \ar[rd]_-{\fn{F}\hat q} 
  \ar[dd]_-{p_X} 
\\
& \fn{F}P_W 
  \ar@{..>}[rr]_-{\fn{F} h} 
  \ar[dd]^-{p_W} 
& 
& \fn{F}Y 
  \ar[dd]^-{p} 
\\ 
  X 
  \ar@{..>}[rrrd]^-{f'} 
  \ar[rd]_-{q} 
\\
& W 
  \ar@{..>}[rr]^-{h'} 
&
& Z 
}\]
where \fun{p}{\fn{F}Y}{Z} is a quotient arrow for $\lift{F}_{Y,Y}(\eqrels)$ in $\RDoc$, 
$f'$ and $h'$ are uniquely determined by the universal property of the quotient arrows $p_X$ and $q$, respectively, 
and $h$ is determined by projectivity of $\fn{F}P_W$ and fullness of $\fn{F}$. 
Note that the upper triangle does not commute. 
However, the above diagram uniquely determines the arrow \fun{\eqc{h}}{\ple{P_W,\eqrelr_W}}{\ple{Y,\eqrels}} in the base of $\EQR\SDoc$ and from 
\begin{align*} 
\lift{F}_{P_X,Y}(\gr{\hat q}\rcomp\gr{h}\rcomp\eqrels) 
  & = \gr{\fn{F}\hat q}\rcomp\gr{\fn{F}h}\rcomp\gr{p}\rcomp\gr{p}\rconv 
    = \gr{p_X}\rcomp\gr{q}\rcomp\gr{h'}\rcomp\gr{p}\rconv 
    = \gr{p_X}\rcomp\gr{f'}\rcomp\gr{p}\rconv
\\
  & = \gr{\fn{F}f}\rcomp\gr{p}\rcomp\gr{p}\rconv 
    = \lift{F}_{P_X,Y}(\gr{f}\rcomp\eqrels) 
\end{align*} 
we deduce that $\gr{\hat q}\rcomp\gr{h}\rcomp\eqrels = \gr{f}\eqrels$, because $\lift{F}$ is an isomorphism, 
thus proving that $\eqc{h}\circ\eqc{\hat q} = \eqc{f}$. 
Therefore, we conclude that $\eqc{\hat q}$ is a quotient arrow for $\hat\eqrelr$ and so $G$ preserves quotients. 

We conclude by showing that $F^\sharp$ and $G$ are pseudoinverse to each other. 
Recall that $F^\sharp = Q^\RDoc\circ\EQR{F}$, where 
\oneAr{Q^\RDoc}{\EQR\RDoc}{\RDoc} is the component at $\RDoc$ of the counit of the biadjunction $\REQFun\dashv\EQRFun$ of \cref{prop:eqc-mnd}. 
Hence, for every object \ple{X,\eqrelr} in the base of $\EQR\SDoc$, there is a quotient arrow \fun{q_\ple{X,\eqrelr}}{\fn{F}X}{\fn{Q^\RDoc}\ple{\fn{F}X,\lift{F}_{X,X}(\eqrelr)}}, where $\fn{Q^\RDoc}\ple{\fn{F}X,\lift{F}_{X,X}(\eqrelr)} = \fn{F^\sharp}\ple{X,\eqrelr}$, and, 
for every $\relr \in \EQR\SDoc(\ple{X,\eqrelr},\ple{Y,\eqrels})$, we have 
$\lift{F^\sharp}_{\ple{X,\eqrelr},\ple{Y,\eqrels}}(\relr) = \gr{q_\ple{X,\eqrelr}}\rconv\rcomp\lift{F}_{X,Y}(\relr)\rcomp\gr{q_\ple{Y,\eqrels}}$. 
We have an invertible 2-arrow \twoAr{\theta}{F^\sharp\circ G}{\Id_\RDoc} given by the following diagram, for every object $X$ 
\[\xymatrix{
  \fn{F}P_X 
  \ar[d]_-{q_\ple{P_X,\eqrelr_X}}
  \ar[rd]^-{p_X} 
\\
  \fn{F^\sharp}\ple{P_X,\eqrelr_X} 
  \ar@{..>}[r]^-{\theta_X}_-{\cong} 
& X 
}\]
where the arrow $\theta_X$ exists and is an isomorphism because both $p_X$ and $q_\ple{P_X,\eqrelr_X}$ are quotient arrows for $\lift{F}_{P_X,P_X}(\eqrelr_X)$. 
On the other hand, 
for an object \ple{X,\eqrelr} in the base of $\EQR\SDoc$, we have the arrows $f_\ple{X,\eqrelr}$ and $g_\ple{X,\eqrelr}$ as in the following diagram, where $W = \fn{F^\sharp}\ple{X,\eqrelr}$ 
\[\xymatrix{
  \fn{F}P_W 
  \ar[rd]_-{p_W}
  \ar@{..>}@/^5pt/[rr]^-{\fn{F}f_\ple{X,\eqrelr}} 
& 
& \fn{F}X 
  \ar[ld]^-{q_\ple{X,\eqrelr}} 
  \ar@{..>}@/^5pt/[ll]^-{\fn{F}g_\ple{X,\eqrelr}} 
\\
& W 
}\]
determined by the fact that both $\fn{F}X$ and $\fn{F}P_W$ are $\RDoc$-projective and $\fn{F}$ is full. 
As already observed, these data uniquely determine the arrows 
\fun{\eqc{f_\ple{X,\eqrelr}}}{\ple{P_W,\eqrelr_W}}{\ple{X,\eqrelr}} and \fun{\eqc{g_\ple{X,\eqrelr}}}{\ple{X,\eqrelr}}{\ple{P_W,\eqrelr_W}} in the base of $\EQR\SDoc$, thus they are inverse to each other. 
Therefore, the 2-arrow \twoAr{\phi}{G\circ F}{\ID_{\EQR{\RProj\RDoc}}} given by $\phi_\ple{X,\eqrelr} = \eqc{f_\ple{X,\eqrelr}}$ is well defined and invertible, as needed.
\end{proof}

If \fun\RDoc{\bop\CC}{\Pos} is a relational doctrine and \fun{F}{\D}{\CC} is a functor, 
we denote by $F^\star\RDoc$ the change of base along $F$, that is, the functor $\RDoc\circ\bop{F}$. 
Furthermore, if \D is a full subcategory of \CC, we denote by $I_\D$ the associated inclusion functor. 

\begin{corollary}\label[cor]{cor:proj-obj}
Let \fun\RDoc{\bop\CC}{\Pos}  be an extensional relational doctrine with quotients and \ct{G} a full subcategory of \CC. 
Then, \ct{G} is an $\RDoc$-projective cover 
if and only if 
$\RDoc$ is equivalent to $\EQR{I_\ct{G}^\star\RDoc}$ in $\EQRDtn$.
\end{corollary} 
\begin{proof}
It follows from \cref{thm:proj-obj} by considering the obvious fully-faithful 1-arrow from $I_\ct{G}^\star\RDoc$ to $\RDoc$. 
\end{proof}

In other words, \cref{cor:proj-obj} ensures that 
relational doctrines obtained by the extensional quotient completion are, up to equivalence, extensional relational doctrines with quotients and enough projectives. 
 Note that, thanks to \refItem{ex:rel-doc-eqc}{span}, \cref{cor:proj-obj} has as a special case 
the standard characterization of the exact completion of a category with weak finite limits in terms of regular projective covers.  Similarly, after \cref{rem:MillyPino}, another instance of \cref{cor:proj-obj} is the characterisation of those elementary and existential doctrines that are an elementary quotient completion given in \cite{MaiettiPR24}.

The following is another example that is not capture by either exact  categories or by existential and elementary doctrines.
\begin{example}
Here we use \cref{cor:proj-obj} to show that categories of assemblies over an appropriate partial applicative structure are the base of a relational doctrine which is an extensional quotient completion; when  the applicative structure is a partial combinatory algebra, we find as instance the characterisation stated in Theorem 4.3 of \cite{lmcs:4302}. 
 Recall from \cite{tomita:LIPIcs.CSL.2021.38} that a BI-algebra is a tuple $A=\ple{\Car{A},\app, B,I}$ where 
$\app$ is a partial binary operation on $\Car{A}$ and $I,B\in \Car{A}$ are such that, for all $a,b,c\in\Car{A}$, 
$I\app a = a$ and $B\app a \app b \app c = a\app (b\app c)$.\footnote{To lighten the notation, we assume that $\app$ associates to the left, i.e., $a\app b \app c$ means $(a\app b)\app c$. } 
As it is customary when working with partial operations, 
here the equality sign denotes Kleene's equality, hence the two equaitons above means that
$I\app a$ is always defined and equal to $a$ and 
$B\app a \app b \app c$ is defined whenever so is $a\app (b\app c)$ and in this case they are equal. 
Examples of BI-algebras are BCI-algebras \cite{AbramskyHS02} and, therefore, partial combinatory algebras \cite{OostenJ:reaait}). 
An assembly over a BI-algebra $A$ is a pair $(X,\alpha)$ where $X$ is a set and $\fun{\alpha}{X}{\pw {\Car{A}}}$ is a function such that each $\alpha(x)$ is not empty; 
if moreover each $\alpha(x)$ is a singleton the assembly is said partitioned (see
\cite{CarboniA:somfcr}). 
A morphism of assemblies $\fun{f}{(X,\alpha)}{(Y, \beta)}$ is a function $\fun{f}{X}{Y}$ such that there is $r\in A$ that realises $f$, 
i.e. for every $x\in X$ and every $a\in \alpha(x)$, $r\app a$ is defined and belongs to $\beta(f(x))$.
Using the combinators $I$ and $B$, we can show that these morphisms are closed under identities and composition, so  assemblies and  their morphisms form a category $\ct{Asm}_{A}$  \cite{tomita:LIPIcs.CSL.2021.38}. 
The relational doctrines $\Rel_U$ obtained by the change of base of $\fun{\Rel}{\bop{\ct{Set}}}{\Pos}$ along the forgetful functor $\fun{U}{\ct{Asm}_{A}}{\ct{Set}}$ is a relational quotient completion. 
To see this, first observe that $\Rel_U$ is trivially extensional and it has quotients because 
a $\Rel_U$-equivalence relation  over $(X,\alpha)$ is just an equivalence relation $\eqrelr\subseteq X\times X$ and this  has a quotient arrow $\fun{q}{\ple{X,\alpha}}{\ple{X_\eqrelr,\alpha_\eqrelr}}$ where $\fun{q}{X}{X_\eqrelr}$ is the usual quotient projection in $\Set$ and $\alpha_\eqrelr(\eqc{x})$  is the union of all $\alpha(x')$ for $x'\eqrelr x$. 
Using the Axiom of Choice, we show that partitioned assemblies form a $\Rel_U$-projective cover. 
Consider an arrow $\fun{f}{(Y,\beta)}{(X/\eqrelr, \alpha_\eqrelr)}$, realised by $r$, whose domain is partitioned. 
With a slight abuse of notation, we identify $\beta(y)$ with the unique element it contains for every $y \in Y$. 
Since for every $y \in Y$ we have that $r\app \beta(y)$ is defined and belongs to $\alpha_\eqrelr(f(y))$, by definition of $\alpha_\eqrelr$, we know that there is $x\eqrelr f(y)$ such that $r\app \beta(y) \in \alpha(x)$. 
By the axiom of choice, we have a function 
$\fun{g}{Y}{X}$ such that $q\circ g = f$ and $r\app\beta(y) \in \alpha(g(y))$, for all $y\in Y$, thus proving that 
$\fun{g}{\ple{Y,\beta}}{\ple{X,\alpha}}$ is a morphism of assemblies. 
Note that if $\ple{Y,\beta}$ were not partitioned, we could not make such choice. 
. Finally note that every $(X,\alpha)$ determines a partitioned assembly $(\widehat{\alpha},\pi_\alpha)$, where  $\widehat{\alpha}$ is the disjoint union of the $\alpha(x)$ and $\pi_\alpha$ maps $(x,a)$ to $a$. The partitioned assemblies $(\widehat{\alpha},\pi_\alpha)$ covers $(X,\alpha)$ via the map that sends $(x,a)$ to $x$, which is a quotient map with respect to the relation $(x,a)\sim (x',a')$ if $x=x'$.
\end{example}

\begin{remark} 
In \refItem{ex:rel-doc-eqc}{vrel} we have observed that the relational doctrine of bimodules over separated metric spaces, i.e. $\fun{\BM}{\bop{\Met}}{\Pos}$, is the quotient completion of the relational doctrine of metric relations $\fun{\VRel\RPos}{\bop\Set}{\Pos}$, so $\BM$ has enough $\BM$-projectives. 
For this to hold, it is crucial that distances take values in $[0,\infty]$. 
Indeed, let us suppose, just for this remark, that $\Met$ is instead the category of  separated metric spaces where all points are at a finite distance, i.e., pairs $\ple{\Car{X},\dist_X}$ where $\dist_X$ takes values in $[0,\infty)$.\footnote{Note that this coincides with the usual textbook definition of metric spaces.} 
Then, we can defined a relational doctrine over them in the same way as we did in \refItem{ex:rel-doc}{met}, but considering bimodules  again taking values only in $[0,\infty)$. 
This doctrine has quotients, whose underline arrows are again surjective functions, 
but it is not a quotient completion as it does not have enough projectives. 
Indeed, consider a metric space $X$ then, for every $r\in (0,\infty)$ let $X_r$ be the space with $\Car{X_r}=\Car{X}$ and $\dist_{X_r}(x,x') = \dist_X(x,x') + r$ when $x\ne x'$. 
In the diagram 
\[
\xymatrix{
& X_r
  \ar[d]^-{\id_X} 
\\
  X
  \ar[r]_-{\id_X}
& X
}
\] 
the vertical arrow is a quotient, so if $X$ is projective, there is an arrow 
\fun{f}{X}{X_r} 
making the diagram commute. 
Viewed as a function, $f$ is necessary the identity on  the set $\Car{X}$, 
hence, because $f$ is non-expansive, for every $r\in(0,\infty)$ and all $x,x'\in\Car{X}$ with $x\ne x'$, we have
$\dist_X(x,x') \geq \dist_{X_r}(x,x') = \dist_X(x,x') + r$, which is impossible. 
This proves that no such metric space is projective. 
Note that this issue does not arise if we allow distances to take values in $[0,\infty]$ as $\infty + r = \infty$ for all $r \in (0,\infty)$.
\end{remark}

\section{Doctrines of algebras as extensional quotient completions} 
\label{sect:algebra}

A key result about the exact completion shows that the category of algebras for a monad on an exact category satisfying a form of the Rule of Choice is equivalent to the exact completion of the full subcategory of free algebras, that is, the Kleisli category \cite{Vitale94}. 
In this section, we show how we can extend this result in the context of relational doctrines. 
To this end, we first prove a variant of this result (\cref{thm:eqc-mnd}) that does not require any form of choice principle, provided that the considered monads preserve quotients. 
Then, we introduce the necessary notions concerning choice principles in relational doctrines, which allow us to finally prove the result (\cref{thm:eqc-mnd-choice}), by observing that, when a sufficiently strong choice principle holds, every monad preserve quotients.

Let us start by recalling what are monads on relational doctrines and how to construct doctrines of algebras. 
Let \fun{\RDoc}{\bop\CC}{\Pos}  be a relational doctrine and $\mnd = \ple{T,\eta,\mu}$ a monad on $\RDoc$ in $\RDtn$,  see \cite{Street72} for  the general 2-categorical definition and \cite{DagninoP25tac} for the instantiation on relational doctrines. 
We can define the Eilenberg-Moore  doctrine \fun{\RDoc^\mnd}{\bop{\CC^{\fn\mnd}}}{\Pos} whose base is the Eilenberg-Moore category for the monad 
$\fn{\mnd} = \ple{\fn{T},\eta,\mu}$ on \CC and relations  from an algebra \ple{X,a} to an algebra \ple{Y,b} are defined as follows
\[
\RDoc^\mnd(\ple{X,a},\ple{Y,b})=\{\relr\in\RDoc(X,Y)\mid \gr{a}\rconv\rcomp\lift{T}_{X,Y}(\relr)\rcomp\gr{b}\order \relr\}
\]
Namely, these are relations closed under the operations of the algebras they relate. The action on arrows is the restriction of that of $\RDoc$, that is $\RDoc^\mnd_{f,g}=\RDoc_{f,g}$. 
 For instance, when $\RDoc$ is the relational doctrine $\Rel$ of set-theoretic relations and 
$\mnd$ is the monad on $\Set$ whose algebras are monoids, extended to relations via the Barr extension~\cite{Barr70}, 
then a relation $\relr$ between monoids $\ple{X,\ast,e}$ and $\ple{X',\ast',e'}$ is a congruence, that is a subset 
$\relr\subseteq X \times X'$ such that $\ple{e,e'}\in\relr$ and, if $\ple{x,x'}\in\relr$ and $\ple{y,y'}\in \relr$, then $\ple{x\ast  y,x'\ast' y'}\in\relr$.

For a full subcategory $\D$ of $\CC$, we denote by $\D_{\fn{\mnd}}$ the full subcategory of the Eilenberg-Moore category $\CC^{\fn\mnd}$ spanned by free algebras generated by objects in $\D$, i.e., algebras of the form \ple{\fn{T}X,\mu_X} for $X$ and object of \D. 
We also write \fun{K_\D}{\D_{\fn\mnd}}{\CC^{\fn\mnd}} for the inclusion functor.
Note that when \D coincides with \CC, the category $\D_{\fn\mnd}$ is the Kleisli category of the monad $\fn\mnd$. 

There is an obvious forgetful 1-arrow \oneAr{U}{\RDoc^\mnd}{\RDoc} in \RDtn which has a left  adjoint \oneAr{F}{\RDoc}{\RDoc^\mnd}, where $\fn{F}X = \ple{\fn{T}X,\mu_X}$ and $\lift{F}_{X,Y}(\relr) = \lift{T}_{X,Y}(\relr)$. 
The component of the counit at an algebra \ple{X,a} is \fun{\epsilon_\ple{X,a}}{\ple{\fn{T}X,\mu_X}}{\ple{X,a}} given by $\epsilon_\ple{X,a} = a$. 
The arrow $\fn{U}\epsilon_\ple{X,a} = a$ is a split epimorphism, so, by the next proposition, when $\RDoc$ is extensional, it is a quotient arrow in $\RDoc$. 

\begin{proposition}\label[prop]{prop:surj-quot-split}
Let $\RDoc$ be an extensional relational doctrine. 
An arrow is a split effective descent  quotient arrow if and only if it is a split epimorphisms. 
\end{proposition}
\begin{proof}
The left-to-right implication is trivial. 
For the other direction, let \fun{q}{X}{Y} be a split epimorphism with section \fun{s}{Y}{X}. 
Consider an arrow \fun{f}{X}{Z} such that $\gr{q}\rcomp\gr{q}\rconv \order \gr{f}\rcomp\gr{f}\rconv$, hence we get 
$\gr{q}\rcomp\gr{s}\rcomp\gr{f} \order \gr{f}$, which implies $\gr{q}\rcomp\gr{s}\rcomp\gr{f} = \gr{f}$ by \cref{prop:fun-ord}. 
Therefore, by extensionality, we derive $f\circ s \circ q = f$ and, since $q$ splits,  it is an epimorphism and so $f\circ s$ is unique, proving that $q$ is a quotient arrow. It is trivially effective and  by \cref{prop:splitepi}  it is also $\RDoc$-surjective.
\end{proof} 

A consequence of \cref{prop:surj-quot-split} is that split quotients are absolute, in the sense that they are preserved by any morphisms of extensional doctrines.

\begin{proposition}\label[prop]{prop:split-one-arrow}
For every morphism of extensional relational doctrines  \oneAr{F}{\RDoc}{\SDoc} and every split effective descent  quotient arrow $\fun{q}{X}{W}$, the arrow $\fun{\fn{F}q}{\fn{X}}{\fn{W}}$ is a split effective descent  quotient arrow. 
\end{proposition}
\begin{proof}
Every functor maps retractions to retractions, so $\fn{F}q$ splits, then it is $\SDoc$-surjective by \cref{prop:splitepi}. The claim follows by \cref{prop:surj-quot-split}.
\end{proof}

\subsection{Quotient preserving monads}
\label{sect:quot-mnd}

Consider now an extensional relational doctrine with quotients, i.e., an object of \EQRDtn, and 
a  monad $\mnd$ on it in \EQRDtn. 
This means that $\mnd$ is a monad on $\RDoc$ in $\RDtn$ such that the underlying 1-arrow preserves quotients. 
A key fact (proved in \cite{DagninoP25tac}) is that in this case the Eilenberg-Moore doctrine $\RDoc^\mnd$ is extensional and has quotients, so it is an object of \EQRDtn. 
Moreover, the forgetful 1-arrow $\oneAr{U}{\RDoc^\mnd}{\RDoc}$ preserves quotients and so does its left adjoint $\oneAr{F}{\RDoc}{\RDoc^\mnd}$. 
Therefore, this adjunction  lies in \EQRDtn. 

The next proposition shows that the forgetful 1-arrow $U$ actually  ``creates' quotients, that is, 
quotient arrows in $\RDoc^\mnd$ are exactly the homomorphisms  whose underlying arrow is a quotient in $\RDoc$. 

\begin{proposition}\label[prop]{prop:mnd-quot-reflect}
Let $\RDoc$ be an extensional relational doctrine with quotients and $\mnd$ a monad on it in \EQRDtn. 
If \fun{q}{\ple{X,a}}{\ple{Y,b}} is an algebra homomorphism such that \fun{q}{X}{Y} is a quotient arrow in $\RDoc$, 
then $q$ is a quotient arrow in $\RDoc^\mnd$ as well. 
\end{proposition}
\begin{proof}
Since $\mnd = \ple{T,\eta,\mu}$ is in \EQRDtn, the functor $T$ preserves quotients. 
Then, we get the thesis by the following commutative diagram in the base of $\RDoc$. 
\[\xymatrix{
  \fn{T}X 
  \ar[rrd]^-{\fn{T}f} 
  \ar[rd]_-{\fn{T}q}
  \ar[dd]_-{a} 
\\ 
& \fn{T}Y 
  \ar[r]_-{\fn{T}h}
  \ar[dd]^-{b} 
& \fn{T}W 
  \ar[dd]^-{c} 
\\
  X 
  \ar[rrd]^-{f} 
  \ar[rd]_-{q}
\\
& Y 
  \ar[r]^-{h} 
& W 
}\]
where $h$ uniquely exists because $q$ is a quotient arrow in $\RDoc$. 
\end{proof}

By combining 
\cref{prop:surj-quot-split,prop:mnd-quot-reflect} we get that the components of the counit \fun{\epsilon_\ple{X,a}}{\ple{\fn{T}X,\mu_X}}{\ple{X,a}} are quotient arrows in $\RDoc^\mnd$, thus showing that every algebra is a quotient of a free algebra. 
We are now ready to prove  our main result. 

\begin{theorem}\label[thm]{thm:eqc-mnd} 
Let \fun\RDoc{\bop\CC}{\Pos} be an extensional relational doctrine with quotients and $\ct{G}$ a subcategory of \CC. 
The following are equivalent. 
\begin{enumerate}
\item\label{thm:eqc-mnd:1}
\ct{G} is an $\RDoc$-projective cover.
\item\label{thm:eqc-mnd:2}
For every monad $\mnd = \ple{T,\eta,\mu}$ on $\RDoc$ in \EQRDtn, $\ct{G}_{\fn\mnd}$ is an $\RDoc^\mnd$-projective cover.
\end{enumerate}
\end{theorem} 
\begin{proof}
The fact that \cref{thm:eqc-mnd:2} implies \cref{thm:eqc-mnd:1} is immediate by considering the identity monad on $\RDoc$. 
Towards a proof of the other direction, 
consider a monad $\mnd = \ple{T,\eta,\mu}$ on $\RDoc$ in \EQRDtn. 
Since every object in $\ct{G}_{\fn\mnd}$ is the image through a left adjoint of an object in \ct{G}, which is  an $\RDoc$-projective cover by hypothesis, 
by \cref{prop:proj-obj-ladj}, we get that all objects in $\ct{G}_{\fn{\mnd}}$ are $\RDoc^\mnd$-projective. 
Consider then an algebra \ple{X,a}. 
Since \ct{G} is an $\RDoc$-projective cover by hypothesis, we know there is a quotient arrow \fun{q_X}{P_X}{X} from an object $P_X$ in \ct{G}. 
Then, the composition 
$\xymatrix{\ple{\fn{T}P_X,\mu_{P_X}}\ar[r]^-{\fn{T}q} & \ple{\fn{T}X,\mu_X}\ar[r]^-{\epsilon_{X,a}} & \ple{X,a}}$
is a quotient arrow by \cref{prop:mnd-quot-reflect} and the fact that quotient arrows compose. 
\end{proof}

Combining \cref{cor:proj-obj,thm:eqc-mnd}, we deduce that 
the relational doctrine of algebras for a quotient preserving monad on   a relational doctrine with a projective cover is the extensional quotient completion  of its restriction to free algebras with projective generators. 
Since relational doctrines with a projective cover are exactly those obtained through the extensional quotient completion, 
this result applies to all quotient preserving monads  on such doctrines, thus providing a wide range of examples. 
 Indeed, the extensional quotient completion, being a 2-functor, maps monads in $\RDtn$ to monads in $\EQRDtn$, that is, 
if $\mnd = \ple{T,\eta,\mu}$ is a monad on $\RDoc$, then $\REQFun(\mnd) = \EQR(\mnd) = \ple{\EQR{T},\EQR\eta,\EQR\mu}$ is a quotient preserving monad on $\EQR\RDoc$. 
It is not difficult to see that the doctrine of algebras $(\EQR{\RDoc})^{\EQR{\mnd}}$ is equivalent to $\EQR{\RDoc^\mnd}$, namely, the extensional quotient completion maps doctrines of algebras to doctrines of algebras. 

\begin{example}
Following the observation above, 
to construct a quotient preserving monad on the doctrine of bimodules over separated metric spaces 
$\fun\BM{\bop\Met}{\Pos}$, 
it suffices to define a monad on the doctrine of metric relations 
$\fun{\VRel\RPos}{\bop\Set}{\Pos}$. 
This essentially amounts to extending  a monad on \Set to metric relations in a compatible way. 
As an example, let us consider the list monad $\ple{\fn{L},\eta,\mu}$ on $\Set$, where, for every set $X$, 
$\fn{L} X = X^\star$ is the set of finite lists of elements in $X$, the function
$\fun{\eta_X}{X}{\fn{L}X}$ maps each $x$ in $X$ to the singleton list containing only $x$, and 
$\fun{\mu_X}{\fn{L}\fn{L}X}{\fn{L}X}$ flattens a list of lists into a single one by concatenating all of them. 
Now, for all sets $X$ and $Y$ we define a monotone function 
$\fun{\lift{L}_{X,Y}}{\VRel\RPos(X,Y)}{\VRel\RPos(\fn{L}X,\fn{L}Y)}$
as follows: 
given $\fun{\relr}{X\times Y}{[0,\infty]}$ we set 
\[ 
\lift{L}_{X,Y}(\relr)(x_1\ldots x_n, y_1,\ldots,y_k) = 
\begin{cases}
  \sum_{i = 1}^n \relr(x_i,y_i) & \text{if } n = k \\
  \infty  &\text{otherwise} 
\end{cases}
\]
It is easy to check that 
$L = \ple{\fn{L},\lift{L}}$ is a 1-arrow on $\VRel\RPos$ in \RDtn and 
$\Lmnd = \ple{L,\eta,\mu}$ is a monad on the same doctrine. 
The doctrine of algebras for $\Lmnd$ is a doctrine of metric congruences over the category of monoids in \Set. 
Then, through the relational quotient completion, we get a quotient preserving monad on $\BM$ whose doctrine of algebras consists of congruence bimodules over monoids in the monoidal category $\Met$, which by \cref{thm:eqc-mnd} has a projective cover. 

This procedure can actually be applied to several monads on the category of sets, because many of them admit a metric extension. 
Indeed, Clementino and Tholen~\cite{ClementinoT14} provide a complete characterization of such monads, which include among the others the powerset monad and the distribution  monad.
\end{example}

\begin{example}
We define the monad $\Kmnd = \ple{K,\eta,\mu}$ on $\fun{\BM}{\bop\Met}{\Pos}$ whose algebras will be separated metric spaces with a $k$-Lipschitz endomap. 
Recall that $\ple{\fn{K},\eta,\mu}$ is a monad on the base category $\Met$ and, as discussed in \cite{Adamek22}, this can be presented by a quantitative equational theory. 
Let us first introduce a bit of notation. 
For a separated metric space $X$ and a positive real number $r$ we rite 
$rX$ for the separated metric space defined by 
$\Car{rX} = \Car{X}$ and $\dist_{rX}(x,x') = r\cdot\dist_X(x,x')$. 
Notice that an $r$-Lipschitz map from $X$ to $Y$ is nothing but a  non-expansive map (i.e., an arrow in $\Met$) from $rX$ to $Y$. 
Recall that the category $\Met$ has small coproducts: 
for a family $(X_i)_{i \in I}$ of separated metric spaces, we have 
$\Car{\coprod_{i \in I} X_i} = \coprod_{i \in I} \Car{X_i}$ and 
\[
\dist_{\coprod_{i \in I} X_i}(\ple{i,x},\ple{j,y}) = 
\begin{cases}
  \dist_{X_i}(x,y) & \text{if } i = j \\
  \infty & \text{otherwise}
\end{cases}
\]
Fix a positive real number $k$ and  consider the functor $\fun{\fn{K}}{\Met}{\Met}$ 
\[ \fn{K} X = \coprod_{n \in \N} k^nX \] 
Note that $\Car{\fn{K}X} = \N \times \Car{X}$. The additive monoid structure of $\N$ induces the monad structure of $\fn{K}$: the unit and the multiplication of the monad are 
\[
\eta_X(x) = \ple{0,x} \qquad 
\mu_X(\ple{n, \ple{m,x}}) = \ple{n + m, x} 
\]
In other words the monad $\ple{\fn{K},\eta,\mu}$ on $\Met$ is a lifting of the output monad $\N \times \blank$ on $\Set$ along the forgetful functor $\Met\to\Set$. 
The algebras for this monad are exactly separated metric spaces with a $k$-Lipschitz endomap. 
Indeed, an arrow $\fun{a}{\fn{K}X}{X}$ is determined by a family of arrows 
$\fun{a_n}{k^nX}{X}$, for all $n \in \N$, and 
the commutative triangle with the unit ensures that $a_0$ is the identity on $X$  and 
the commutative square with the multiplication forces $a_{n+1}$ to be equal to $a_1 \circ a_n$, noting that 
$a_n$ is also an arrow from $k^{n+1}X$ to $kX$. 
Therefore, the entire algebra structure is completely determined by the arrow $\fun{a_1}{kX}{X}$, which is a $k$-Lipschitz endomap on $X$. 

To get a monad on the doctrine $\BM$ we still have to define the action on bimodules. 
This can be done  as follows: 
given $\relr \in \BM(X,Y)$ we set 
\[
\lift{K}_{X,Y}(\relr)(\ple{n,x},\ple{m,y}) = 
\begin{cases}
  k^n\cdot \relr(x,y) & \text{if } n = m \\
  \infty & \text{otherwise} 
\end{cases}
\]
It is not difficult to check that $K = \ple{\fn{K},\lift{K}}$ is a 1-arrow on $\BM$ in $\RDtn$ and that 
$\Kmnd = \ple{K,\eta,\mu}$ is a monad on $\BM$ in $\RDtn$. 
To get a monad in $\EQRDtn$ we have to verify that $K$ preserves quotients. 
Recall from \refItem{ex:rel-doc-q}{met} that quotients in $\BM$ are arrows whose underlying function is surjective. 
Therefore, we have to show that, if $\fun{f}{X}{Y}$ is a surjective non-expansive map, so is $\fn{K}f$, but this is  straightforward because $\fn{K}f = \id_{\N}\times f$. 

This provides an example of a quotient preserving monad on $\BM$ which is not obtained through  the extensional quotient completion. 
Moreover, it suggests that quantitative equational theories could determine monads on $\BM$ in a natural way. 
However, we leave  a full development of this observation for future work. 

\end{example}

\subsection{Factorization, balancing and choice} 
\label{sect:quot-factor} 
 
Choice principles can take many forms, depending on the specific setting where one works and on the way they are formulated. 
The goal of this section is to develop the appropriate variants of choice principles in relational doctrines 
that will allow us to state and prove \cref{thm:eqc-mnd-choice}. 
To achieve this, we need to study how these principles are related to quotients and projective objects and this in turn requires 
some results concerning factorizations of arrows in relational doctrines with quotients. 

The first observation is that quotient and $\RDoc$-injective arrows determine an orthogonal factorization system, as the following propositions show. 

\begin{proposition}\label[prop]{prop:qi-lift}
Let \fun\RDoc{\bop\CC}{\Pos} be a relational doctrine with quotients. 
For every commutative square in \CC 
\[\xymatrix{
  X \ar[r]^-{f} \ar[d]_-{q} 
& Y \ar[d]^-{i} \\
  W \ar[r]_-{g} \ar@{..>}[ru]^-{h} 
& Z 
}\]
where $q$ is a quotient arrow and $i$ is $\RDoc$-injective, 
there exists a unique arrow $h$ making the two triangles commute. 
\end{proposition}
\begin{proof}
Since the square commutes and $i$ is $\RDoc$-injective, we have the following inequalities:
\[ 
\gr{q}\rcomp\gr{q}\rconv 
  \order \gr{q}\rcomp\gr{g}\rcomp\gr{g}\rconv\rcomp\gr{q}\rconv 
  = \gr{f}\rcomp\gr{i}\rcomp\gr{i}\rconv\rcomp\gr{f}\rconv 
  \order \gr{f}\rcomp\gr{f}\rconv 
\]
Hence, since $q$ is a quotient arrow, we get a unique arrow \fun{h}{W}{Y} such that 
$f = h\circ q$. 
Moreover $q$ is epic, as it is a quotient arrow, so from 
$g\circ q = i\circ f = i\circ h\circ q$ follows that $g = i\circ h$. 
\end{proof}

\begin{proposition}\label[prop]{prop:qi-factor}
Let \fun\RDoc{\bop\CC}{\Pos} be a relational doctrine with quotients and \fun{f}{X}{Y} an arrow in \CC. 
Then, $f$ factors as $i\circ q$ where $q$ is the quotient arrow of the kernel of $f$ and $i$ is injective. 
\end{proposition}
\begin{proof}
Let \fun{q}{X}{W} be a quotient arrow for the $\RDoc$-equivalence $\eqrelr = \gr{f}\rcomp\gr{f}\rconv$, i.e., the kernel of $f$. 
By the universal property of quotients, we get an arrow \fun{i}{W}{Y} such that $f = i\circ q$, so 
$\gr{q}\rcomp\gr{i}=\gr{f}$, which implies $\gr{i} = \gr{q}\rconv\rcomp\gr{f}$ because $q$ is $\RDoc$-surjective. 
Hence 
$\gr{i}\rcomp\gr{i}\rconv = \gr{q}\rconv\rcomp\gr{f}\rcomp\gr{f}\rconv\rcomp\gr{q} = \gr{q}\rconv\rcomp\gr{q}\rcomp\gr{q}\rconv\rcomp\gr{q} = \rid_W$ showing that $i$ is injective.
\end{proof}

\begin{corollary}\label[cor]{cor:fac-s}
Let \fun\RDoc{\bop\CC}{\Pos} be a relational doctrine with quotients. 
Then, quotient arrows and injective arrows form an orthogonal factorization system on \CC. Moreover if $\RDoc$ is extensional, the factorization system is proper.
\end{corollary}
\begin{proof}The first part is a direct consequence of \cref{prop:qi-lift} and \cref{prop:qi-factor}. The second part follows by \cref{prop:extomonic}.
\end{proof}

Since quotient arrows are required to be $\RDoc$-surjective, 
\cref{prop:qi-factor} shows that in a relational doctrine with quotients any arrow factors as a $\RDoc$-surjective one followed by a $\RDoc$-injective one. 
However, the fact that the $\RDoc$-surjection is actually a quotient arrow is crucial for carrying out the proofs of \cref{prop:qi-lift,prop:qi-factor}. 
 In fact, in general, $\RDoc$-surjections and $\RDoc$-injections do not form a factorization system, because $\RDoc$-surjections need not be quotient arrows. 
For instance,  in the doctrine $\fun{\BM}{\bop\Met}{\Pos}$, introduced in \refItem{ex:rel-doc}{met}, the $\BM$-surjections are the dense maps, while the quotients are non-expansive maps whose underlying function is surjective. 
Hence, in order to understand when $\RDoc$-surjections and $\RDoc$-injections form a factorization system, we should study when $\RDoc$-surjective and quotient arrows coincide. 
Here it is where the first choice principle comes into play.

\begin{definition}\label[def]{def:balance}
Let \fun{\RDoc}{\bop\CC}{\Pos} be a relational doctrine. 
We say that $\RDoc$ is \emph{balanced} if every $\RDoc$-bijective arrow in \CC is an isomorphism. 
\end{definition}

 Notice that  quotient arrows always have this balancing property: 
if a quotient arrow is $\RDoc$-bijective then it is an isomorphism. 
Therefore, the balancing property of \cref{def:balance} essentially requires that surjective arrows behaves like quotient arrows in this respect. 

\begin{remark}
\label[rem]{rem:cat-balance}
The notion of balancing in relational doctrines is inspired by the analogous concept in plain categories (recall that
a category is balanced when any arrow which is both monic and epic it is an isomorphism) but it is not a generalization of it.
Indeed if $\CC$ is regular, the doctrine $\jmSpan{\CC}$ is balanced as the $\jmSpan{\CC}$-injections are monomorphisms and the $\jmSpan{\CC}$-surjections are regular epimorphisms, but $\CC$ need not be balanced as a category, because not all epimorphisms are regular.
Indeed a regular category is balanced if and only if all epimorphisms are regular.
\end{remark}

\begin{remark} 
The balancing property expresses a form of the unique choice principle. 
Intuitively, it says that whenever for an arrow $\fun{f}{X}{Y}$ we can prove that for every element $y$ in $Y$ there is a unique $x$ in $X$ which is mapped by $f$ into $y$, then $x$ can be picked by a function, namely the inverse of $f$. 
More formally, the balancing property requires that total and functional relations
which are the inverse of the graph of an arrow are actually
the graph of the inverse of such an arrow. 
Notice that the balancing property is actually weaker than the Rule of Unique Choice introduced in \cite{DagninoP25tac}, which requires that all total and functional relations to be the graph of some arrow. 
\end{remark}

\begin{proposition}\label[prop]{prop:quot-surj-balance}
Let $\RDoc$ be a relational doctrine with quotients. 
Then, $\RDoc$ is balanced if and only if every $\RDoc$-surjective arrow is a quotient arrow. 
\end{proposition}
\begin{proof}
To prove the left-to-right implication, consider a $\RDoc$-surjective arrow \fun{s}{X}{Y} and its factorization $s = i\circ q$ as in \cref{prop:qi-factor}. 
Then, from $\gr{i} = \gr{q}\rconv\rcomp\gr{s}$ we deduce 
$ \gr{i}\rconv\rcomp\gr{i} 
  = \gr{s}\rconv\rcomp\gr{q}\rcomp\gr{q}\rconv\rcomp\gr{s} 
  = \gr{s}\rconv\rcomp\gr{s} 
  = \rid_Y$, 
hence $i$ is $\RDoc$-surjective. 
Since $\RDoc$ is balanced, $i$ is an isomorphism and thus $s$ is a quotient arrow. To prove the right-to-left implication, consider a $\RDoc$-bijective arrow \fun{f}{X}{Y}. 
Then, $f$ is in particular $\RDoc$-surjective, so a quotient arrow of $\gr{f}\rcomp\gr{f}\rconv = \rid_X$. But also $\id_X$ is a quotient arrow of $\rid_X$, which make $f$ an isomorphism.
\end{proof}

Therefore (by \cref{cor:fac-s}) in a balanced relational doctrine with quotients 
$\RDoc$-surjective and $\RDoc$-injective arrows form an orthogonal factorization system. 

A way to express  the Rule of Choice in a category is to ask that epimorphisms split. Sometimes one asks that only a certain class of epimorphisms split (e.g. in the internal logic of a regular category it is natural to say that Choice holds if regular epimorphisms split). 
 Notice that  in a balanced category (see \cref{rem:cat-balance}) these two formulations coincide. %
Similarly, in a relational doctrine one can state the Rule of Choice either by requiring that all quotient arrows split or that all surjective arrows split 
and, by \cref{prop:quot-surj-balance}, these two formulations are equivalent when the doctrine is balanced. 
The next proposition relates  these concepts to projective objects. 
\begin{proposition}\label[prop]{prop:proj-obj-choice}
Let \fun{\RDoc}{\bop\CC}{\Pos} be  a relational doctrine with quotients. 
Then, the following hold.
\begin{enumerate}
\item\label{prop:proj-obj-choice:1}
All objects in \CC are $\RDoc$-projective if and only if every quotient arrow splits. 
\item\label{prop:proj-obj-choice:3}
If $\RDoc$ is extensional, then 
$\RDoc$ is balanced and all objects of \CC are $\RDoc$-projective if and only if every $\RDoc$-surjective arrow in \CC splits. 
\end{enumerate}
\end{proposition}
\begin{proof}
\cref{prop:proj-obj-choice:1}. 
Suppose all objects in \CC are $\RDoc$-projective and consider a quotient arrow \fun{q}{X}{W}. 
Since $W$ is $\RDoc$-projective, there is an arrow \fun{s}{W}{X} such that $q\circ s = \id_W$ as neded. 
On the other hand, if all quotient arrows split, given an arrow \fun{f}{X}{Z} and a quotient arrow \fun{q}{Y}{Z}, we know that $q$ has a section \fun{s}{Z}{Y} and so we have $q\circ s \circ f = f$, proving that $X$ is $\RDoc$-projective. \cref{prop:proj-obj-choice:3}. 
Immediate using \cref{prop:proj-obj-choice:1,prop:surj-quot-split,prop:quot-surj-balance}. 
\end{proof}

We conclude the section studying how doctrines of algebras interact with the choice principles introduced so far.

\begin{proposition}\label[prop]{prop:eqc-mnd-balanced} 
Let \fun\RDoc{\bop\CC}{\Pos} be an extensional relational doctrine. 
Then, the following are equivalent
\begin{enumerate}
\item\label{prop:eqc-mnd-balanced:1}
$\RDoc$ is balanced. 
\item\label{prop:eqc-mnd-balanced:2}
For every monad $\mnd = \ple{T,\eta,\mu}$ on $\RDoc$ in $\RDtn$, the doctrine $\RDoc^\mnd$ is balanced.
\end{enumerate}
\end{proposition}
\begin{proof} Recall that $\RDoc^\mnd$ inherits from $\RDoc$ all the relational operations and the reindexing, therefore if $\fun{\ple{Td,d}}{\ple{X,a}}{\ple{Y,b}}$ is $\RDoc^\mnd$-bijective, $d$ is $\RDoc$-bijective, hence an isomorphism in $\CC$. So also $Td$ is an isomorphism.

The converse follows taking as $\mnd$ the identity functor on $\CC$. 
\end{proof}

The following is a stronger form of \cref{thm:eqc-mnd}. 

\begin{theorem}\label[thm]{thm:eqc-mnd-choice} 
Let \fun\RDoc{\bop\CC}{\Pos} be an extensional relational doctrine with quotients. 
Then, the following are equivalent
\begin{enumerate}
\item\label{thm:eqc-mnd-choice:1}
Every quotient arrow in \CC splits. 
\item\label{thm:eqc-mnd-choice:2}
For every monad $\mnd = \ple{T,\eta,\mu}$ on $\RDoc$ in $\RDtn$, the relational doctrine $\RDoc^\mnd$ has quotients and $\CC_{\fn\mnd}$ is a $\RDoc^\mnd$-projective cover.
\end{enumerate}
\end{theorem}
\begin{proof}
Suppose $\RDoc$-quotients split. By \refItem{prop:proj-obj-choice}{1} we know that objects of $\CC$ are $\RDoc$-projective, hence \CC is a $\RDoc$-projective cover of itself. 
Consider a monad $\mnd = \ple{T,\eta,\mu}$ on $\RDoc$ in \RDtn, then 
 $T$ preserves quotient arrows by \cref{prop:split-one-arrow}. This makes $\mnd$ a monad in \EQRDtn. By \cref{thm:eqc-mnd} the category $\CC_{\fn\mnd}$ is an $\RDoc^\mnd$-projective cover. 

Conversely, consider the identity monad on $\RDoc$, whose Eilenberg-Moore doctrine is $\RDoc$ itself and $\CC_{\fn\mnd}$ is $\CC$. Thus all the objects of $\CC$ are $\RDoc$-projective, so the thesis follows by \refItem{prop:proj-obj-choice}{1}. 
\end{proof}

Notice that, in proving the  theorem above, we have also proved that every monad in \RDtn on an extensional relational doctrine with quotients where surjections split is actually a monad in $\EQRDtn$, that is, its underlying 1-arrow preserves quotients. 

\begin{corollary}\label[cor]{cor:eqc-mnd-choicec} 
Let \fun\RDoc{\bop\CC}{\Pos} be an extensional relational doctrine with quotients. 
Then, the following are equivalent
\begin{enumerate}
\item\label{thm:eqc-mnd-choicec:1}
Every $\RDoc$-surjective arrow in \CC splits. 
\item\label{thm:eqc-mnd-choicec:2}
For every monad $\mnd = \ple{T,\eta,\mu}$ on $\RDoc$ in $\RDtn$, the relational doctrines $\RDoc^\mnd$ is  balanced,  has quotients and $\CC_{\fn\mnd}$ is a $\RDoc^\mnd$-projective cover.
\end{enumerate}
\end{corollary}
\begin{proof}
Note that by \cref{prop:proj-obj-choice}  $\RDoc$-surjections split if and only if $\RDoc$ is balanced and quotients split. The claim follows by  \cref{prop:eqc-mnd-balanced} and \cref{thm:eqc-mnd-choice}. 
\end{proof}

\cref{thm:eqc-mnd-choice} almost resemble the analogous result for the exact completion \cite{Vitale94}, but it is not a proper generalization. 
Indeed \cref{thm:eqc-mnd-choice} applies to any relational doctrine and not just to that of jointly monic spans, which is the one behind the exact completion, 
however, it considers monads compatible with relational operations, while the result in \cite{Vitale94} applies to arbitrary monads (not necessarily preserving pullbacks and factorizations).
Nevertheless, we can still recover a result for arbitrary monads on an exact category as the following example shows. 

\begin{example}\label[ex]{ex:eqc-mnd-exact} 
Let \CC be an exact category where regular epimorphisms split
and consider a monad $\mnd = \ple{T,\eta,\mu}$ on \CC. 
The Eilenberg-Moore category $\CC^\mnd$ is exact as well. 
The forgetful functor \fun{U}{\CC^\mnd}{\CC}, being a right adjoint, preserves monomorphisms. 
Moreover, it preserves coequalizers of equivalence relations. 
Hence, it extends to a 1-arrow \oneAr{\JSpan{U}}{\JSpan{\CC^\mnd}}{\JSpan\CC} in \EQRDtn, where $\fn{\JSpan{U}} = U$. 
The functor $U$ has a left adjoint \fun{F}{\CC}{\CC^\mnd} and the component at the algebra \ple{X,a} of the counit of this adjunction is the algebra homomorphism \fun{\epsilon_\ple{X,a}}{\ple{TX,\mu_X}}{\ple{X,a}} given by $\epsilon_\ple{X,a} = a$, which is a coeqalizer and so a regular epimorphism, that is, a quotient arrow in $\JSpan{\CC^\mnd}$. 
Since every regular epimorphism in \CC splits, every object of \CC is $\JSpan\CC$-projective. 
Therefore, by \cref{prop:proj-obj-ladj}, we get that the Kleisli category $\CC_\mnd$ is a $\JSpan{\CC^\mnd}$-projective cover, 
hence, by \cref{cor:proj-obj}, 
we get that $\JSpan{\CC^\mnd}$ is equivalent to the extensional quotient completion of its restriction to $\CC_\mnd$. 
In other words, if \fun{K}{\CC_\mnd}{\CC^\mnd} is the inclusion functor of $\CC_\mnd$ into $\CC^\mnd$, we have that 
$\JSpan{\CC^\mnd}$ is equivalent to $\EQR{K^\star\JSpan{\CC^\mnd}}$. 
Finally, using \cref{l:analogo0}, 
it is not difficult to observe that $K^\star\JSpan{\CC^\mnd}$ is equivalent to $\Span{\CC_\mnd}$, thus recovering the known result in \cite{Vitale94}. 
\end{example}

\bibliographystyle{elsarticle-num} 
\bibliography{biblio}

\end{document}